\documentclass[12pt]{article}
\usepackage{epsfig}
\usepackage{authblk}
\usepackage{epstopdf}
\usepackage{graphics}
\usepackage{amsmath}
\usepackage{amsthm}
\usepackage{amssymb}
\usepackage{array}
\usepackage{subfig} 
\usepackage[font={footnotesize}]{caption}
\usepackage{relsize}
\usepackage[nottoc]{tocbibind}
\usepackage{hyperref}
\usepackage{float}

\usepackage{chngcntr}
\usepackage{wrapfig}

\counterwithin*{equation}{section}
  
\newtheorem{thm}{Theorem}[section] 
\newtheorem{lemma}[thm]{Lemma}  
 
\newtheorem{proposition}[thm]{Proposition}  
\newtheorem{conj}[thm]{Conjecture}
\newtheorem{hyp}[thm]{Hypothesis}
\newtheorem{defn}{Definition}[section]

\newtheoremstyle{named}{}{}{\itshape}{}{\bfseries}{.}{.5em}{\thmnote{#3's }#1}
\theoremstyle{named}

\theoremstyle{definition}
\newtheorem{rmk}{Remark}

\title{Rotating Wave Solutions to Lattice Dynamical Systems I: The Anti-Continuum Limit}
\author{Jason J. Bramburger\\ Division of Applied Mathematics\\ Brown University \\ Providence, Rhode Island 02906\\ USA}
\date{} 
 
\begin{document} 

\maketitle

\begin{center}
{\em This is a post-peer-review, pre-copyedit version of an article published in the Journal of Dynamics and Differential Equations}
\end{center}

\abstract{Rotating waves are a fascinating feature of a wide array of complex systems, particularly those arising in the study of many chemical and biological processes. With many rigorous mathematical investigations of rotating waves relying on the model exhibiting a continuous Euclidean symmetry, this work is aimed at understanding these nonlinear waves in the absence of such symmetries. Here we will consider a spatially discrete lattice dynamical system of Ginzburg-Landau type and prove the existence of rotating waves in the anti-continuum limit. This result is achieved by providing a link between the work on phase systems stemming from the study of identically coupled oscillators on finite lattices to carefully track the solutions as the size of the lattice grows. It is shown that in the infinite square lattice limit of these phase systems that a rotating wave solution exists, which can be extended to the Ginzburg-Landau system of study here. The results of this work provide a necessary first step in the investigation of rotating waves as solutions to lattice dynamical systems in an effort to understand the dynamics of such solutions outside of the idealized situation where the underlying symmetry of a differential equation can be exploited.}

\section{Introduction}

Examples of rotating waves abound in nature and have been an intense area of rigorous mathematical investigation for many decades now. Spiral waves are a particularly important example of rotating waves which present themselves as striking visual patterns, whose formal study dates back to the work of Winfree in chemical reaction theory $\cite{Winfree,Winfree2}$. They have been shown to be associated with many serious phenomena in electrophysiological pathologies. This includes, but is not limited to, cortical spreading depression, hallucinations and ventricular fibrillation $\cite{Beaumont,Gorelova,Huang,Cardiac,KeenerSneyd,Cortical}$. With spiral waves occurring in such circumstances, it follows that they remain an active and intense area of study both in mathematics and throughout the physical sciences. 

Mathematical investigations of rotating waves have highlighted that the underlying symmetry of a differential equation plays a critical role in understanding the dynamics and bifurcations of these waves $\cite{Barkley,Barkley2,Victor, SSW3}$. Typical investigations of spiral waves focus on reaction-diffusion equations for exactly this reason. That is, consider the partial differential equation (PDE) of the form
\begin{equation} \label{RDE}
	\frac{\partial u}{\partial t} = \frac{\partial^2 u}{\partial x^2} + \frac{\partial^2 u}{\partial y^2} + \mathcal{F}(u),  
\end{equation}   
where $u = u(x,y,t): \mathbb{R}^2 \times \mathbb{R} \to \mathbb{R}^n$ and $\mathcal{F}:  \mathbb{R}^n \to  \mathbb{R}^n$, for some $n \geq 1$. Equation ($\ref{RDE}$) possesses an important symmetry property: if $u(x,y,t)$ is a solution to $(\ref{RDE})$ then so is
\begin{equation}
	\tilde{u}(x,y,t) = u(x\cos\theta - y\sin\theta + p_1, x\sin\theta + y\cos\theta + p_2, t), 
\end{equation}  
for any angle $\theta$ and translation $(p_1,p_2)\in\mathbb{R}^2$. These rotations and translations together form the special Euclidean group, often denoted SE$(2)$, and equation $(\ref{RDE})$ precisely is said to be invariant with respect to the action of this group on suitable function spaces. Using this symmetry property of the differential equation we can construct a centre-manifold reduction of the infinite-dimensional partial differential equation to a finite-dimensional system of ordinary differential equations near rotating wave solutions of $(\ref{RDE})$ $\cite{SSW,SSW2}$.  

Euclidean symmetry has been an excellent tool to describe the macroscopic behaviour of rotating waves, but one should note that in reality it is a modelling hypothesis which is, at best, an approximation. That is, bounded domains, heterogeneities and anisotropy are all important in physical models and violate a Euclidean symmetry assumption in the model. This has lead to some investigations of symmetry-breaking perturbations, which have demonstrated small, but measurable, discrepancies between systems with full symmetry and those with broken symmetry \cite{Victor2,Victor3,Victor4}.     

Our work here is part of a larger research program aimed at furthering the understanding of the behaviour of rotating waves in the absence of Euclidean symmetry. More precisely, we aim to determine how the dynamic behaviour of the rotating waves differ when Euclidean symmetry is absent. As an alternative to considering symmetry breaking perturbations, this line of questioning has led to considering a countably infinite system of coupled ordinary differential equations, termed a lattice dynamical system (LDS), of the form  
\begin{equation} \label{LDS}
	\dot{u}_{i,j} = \alpha(u_{i+1,j} + u_{i-1,j} + u_{i,j+1} + u_{i,j-1} - 4u_{i,j}) + f(u_{i,j}),
\end{equation}	
where $u_{i,j} = u_{i,j}(t):\mathbb{R} \to \mathbb{R}^n$, $n\geq 1$, for each $(i,j) \in \mathbb{Z}^2$ and $\dot{x} = dx/dt$. Here $\alpha > 0$ is regarded as the strength of coupling between neighbouring elements in the integer lattice and $f:\mathbb{R}^n \to \mathbb{R}^n$ is a general nonlinearity. One can see that in moving from the partial differential equation context of $(\ref{RDE})$ to that of the lattice dynamical system $(\ref{LDS})$ we have replaced the continuous two-dimensional spatial medium with a grid, or lattice, which moves our problem into a discrete spatial setting. The nearest-neighbour coupling of system $(\ref{LDS})$ can be derived as the leading order of a typical finite-difference approximation of the second order diffusion differential operator. One may also consider more complicated connection topologies, as is being done in some recent works on finite lattices $\cite{DeVille,Udeigwe}$, but since the long term objective is to explore solutions to differential equations in discrete space versus continuous space, the nearest-neighbour connections of system $(\ref{LDS})$ will suffice. 

The most important take-away from system $(\ref{LDS})$ is that it does not satisfy continuous Euclidean symmetry invariance, and therefore provides a basis for the inspection of rotating wave dynamics in the absence of Euclidean symmetry. Although, one should note that system $(\ref{LDS})$ does in fact retain the discrete symmetries of the integer lattice $\mathbb{Z}^2$; that is, discrete translations in both the horizontal and vertical directions along with a four-fold rotational symmetry. The fact that these symmetries are discrete prevents the use of typical methods employed in the continuous spatial setting, but in this work we will see that they can be exploited in other ways to demonstrate the existence of rotating wave solutions to lattice dynamical systems. 

In addition to lattice systems being a prototype for spatial discretizations of PDEs, they have proven extremely useful in describing numerous phenomena irrespective of their continuous space counterparts. LDSÕs arise naturally in various physical settings such as material science, in particular metallurgy, where LDSs have been used to model solidification of alloys $\cite{Cahn2,Cook}$, chemical reactions $\cite{Erneux}$, optics $\cite{Firth}$ and biology; particularly with chains of coupled oscillators arising in models of neural networks $\cite{Ermentrout,ErmentroutKopell2}$. For these applications and many more, LDSs have therefore proven to be an interesting area of research in their own right.

\subsection{Traveling Wave Solutions to LDSs}  

Let us begin by illustrating an important and motivating example. Consider the reaction-diffusion equation in one spatial dimension
 \begin{equation} \label{RDE1D} 
 	\frac{\partial u}{\partial t} = d\cdot\frac{\partial^2 u}{\partial x^2} + u(1- u)(a - u),	
 \end{equation} 
 where $d > 0$, $a \in (0,1)$ and $u = u(x,t) : \mathbb{R}\times\mathbb{R} \to \mathbb{R}$. Finding traveling wave solutions to $(\ref{RDE1D})$ requires determining the existence of a solution of the form $u(x,t) = \phi(x - ct)$, where $c \in \mathbb{R}$ is the wave speed and $\phi:\mathbb{R} \to \mathbb{R}$ is the wave profile, satisfying appropriate boundary conditions. Much work has been done to demonstrate the existence of these desired solutions to equation $(\ref{RDE1D})$, beginning with the pioneering work of Fife and McLeod $\cite{Fife}$. One finds that the solutions $\phi$ not only exist, but further exhibit an explicit dependence between the wave speed, $c$, and the parameter $a$. Particularly, at the critical parameter value $a = 1/2$ the waves have zero speed and thus fail to propagate through the continuous spatial medium. This phenomenon is often referred to as {\em propagation failure}.  
 
The analogous lattice dynamical system to $(\ref{RDE1D})$ is 
\begin{equation} \label{1DLDS}
	\dot{u} = \alpha(u_{i+1} + u_{i-1} - 2u_{i}) + u_{i}(1- u_{i})(a - u_{i}), \ \ \ \ \ i\in\mathbb{Z}. 
\end{equation} 
Traveling wave solutions now take the form $u_{i}(t) = \phi(i - ct)$, again for a wave profile $\phi:\mathbb{R} \to \mathbb{R}$ and appropriate boundary conditions. Searching for traveling wave solutions to $(\ref{1DLDS})$ requires considerably different techniques to that of the continuous spatial medium, with the existence of such solutions being demonstrated most notably by Zinner $\cite{Zinner}$. Here the wave speed has not been explicitly related to the parameter $a$, although Keener has demonstrated that when coupling is sufficiently weak there are entire open regions in parameter space which lead to propagation failure $\cite{Keener}$. 

By moving from one spatial dimension to two spatial dimensions we arrive at further comparisons between the discrete and continuous spatial settings. For example, isotropy of $(\ref{RDE})$ implies that the direction in which a traveling wave propagates does not effect the qualitative dynamics of the wave. This is not necessarily the case in the discrete spatial setting, since it has been shown that the direction of propagation can play a direct role in determining the waves ability to propagate through the discrete spatial medium $\cite{Cahn}$.    

To date there have been numerous studies on the existence and properties of traveling wave solutions to lattice dynamical systems, with a particular emphasis on fronts which fail to propagate through the discrete spatial medium $\cite{Cahn,Elmer,Hupkes,Keener}$. Our work is therefore motivated by the many investigations of traveling waves demonstrating the slight, but measurable, differences in dynamics between continuous and discrete space. That is, it has become apparent that systems which do not satisfy a continuous Euclidean symmetry assumption can in some cases provide qualitatively different traveling wave solutions to the Euclidean invariant case, and hence one is naturally led to question how these investigations can be extended to the study of rotating waves. Therefore this work here aims to further this line of investigation by considering rotational propagation, providing the necessary existence results, similarly to what Zinner has done for traveling waves $\cite{Zinner}$.

\subsection{The Model}

Our investigation begins by considering the simplified cubic Ginzburg-Landau reaction-diffusion system, written in terms of a single complex variable $z(x,y,t):\mathbb{R}^2 \times \mathbb{R} \to \mathbb{C}$, of the form 
\begin{equation} \label{RDELambdaOmega}
	\frac{\partial z}{\partial t} = D\bigg(\frac{\partial^2 z}{\partial x^2} + \frac{\partial^2 z}{\partial y^2}\bigg) + (1 + {\rm i}\omega)z - z|z|^2,	
\end{equation}
where ${\rm i}=\sqrt{-1}$ is the imaginary constant, $D > 0$ is a diffusion coefficient and $\omega\in\mathbb{R}$. These reaction-diffusion equations are well-known to arise as the lowest order perturbation of any reaction-diffusion system near a Hopf bifurcation $\cite{Cohen}$. Systems such as $(\ref{RDELambdaOmega})$ have become an archetype for oscillatory behaviour in reaction-diffusion systems, and most importantly to our study here is that PDEs of this type are well-known to exhibit spiral wave solutions $\cite{Cohen,Greenberg,Kopell2,Troy}$. Therefore system $(\ref{RDELambdaOmega})$ will provide an optimal starting point for the mathematical investigation presented here. 

The spatially discrete analogue of $(\ref{RDELambdaOmega})$ takes the form
\begin{equation} \label{ComplexLattice}
	\dot{z}_{i,j} = \alpha \sum_{i',j'} (z_{i',j'} - z_{i,j}) + (1 + {\rm i}\omega)z_{i,j} - z_{i,j}|z_{i,j}|^2, \ \ \ \ \ \ \ (i,j) \in\mathbb{Z}^2,
\end{equation}
where the sum represents the coupling terms of $(\ref{LDS})$ over all nearest-neighbours of the lattice point $(i,j)$. Since systems of type $(\ref{RDELambdaOmega})$ provided the setting for the first rigorous inspection of spiral wave solutions to PDEs $\cite{Cohen}$, it is therefore natural to consider $(\ref{ComplexLattice})$ for the investigation of spiral waves in LDSs. Although there have been some studies of the systems of type $(\ref{ComplexLattice})$ on finite lattices which allow for the formation of important conjectures about the infinite lattice $\cite{DarkVortex,ErmentroutLambdaOmega}$, the case of an infinite lattice remains to be inspected. By posing our system on an infinite lattice we gain insight into the behaviour of the solution as the size of the lattice grows without bound and understand the behaviour of solutions in the absence of boundary conditions.       

The parameter $\alpha$ is typically referred to as the {\em coupling coefficient} and represents the strength of the effect of the spatial discretization. The limit $\alpha \to \infty$ corresponds to a return to the continuum equation $(\ref{RDELambdaOmega})$, whereas the limit $\alpha \to 0^+$ corresponds to the so-called {\em anti-continuum limit}. In the interest of exploring the dynamics of rotating waves in a fully spatially discrete setting we will investigate solutions in the anti-contiuum limit here.    

To properly analyze system $(\ref{ComplexLattice})$ we write
\begin{equation} \label{ComplexPolarForm}
	z_{i,j} = r_{i,j}e^{{\rm i}(\omega t + \theta_{i,j})}, 
\end{equation}	
for each $(i,j)$, where $r_{i,j} =r_{i,j}(t)$ and $\theta_{i,j} = \theta_{i,j}(t)$. System ($\ref{ComplexLattice}$) can now be written as
\begin{subequations}\label{PolarLattice}
	\begin{equation}\label{PolarLatticeRadial}
		\dot{r}_{i,j} = \alpha\sum_{i',j'} (r_{i',j'}\cos(\theta_{i',j'} - \theta_{i,j}) - r_{i,j}) + r_{i,j}(1 - r_{i,j}^2),
	\end{equation}
	\begin{equation}\label{PolarLatticePhase}
		\dot{\theta}_{i,j} =  \alpha\sum_{i',j'} \frac{r_{i',j'}}{r_{i,j}}\sin(\theta_{i',j'} - \theta_{i,j}), \ \ \ (i,j) \in \mathbb{Z}^2.
	\end{equation}
\end{subequations} 
The goal is to eventually provide a proof of rotating/spiral wave solutions to system $(\ref{ComplexLattice})$ using the polar decomposition $(\ref{PolarLattice})$, and our work here provides a necessary intermediate step to establishing this result. In this work we will restrict ourselves to investigating the phase components $(\ref{PolarLatticePhase})$ in the anti-continuum limit $\alpha \to 0^+$, where one can see that the problem now becomes one of singular perturbation. In Section $\ref{sec:Model}$ we provide the exact coupled phase system we wish to study in this work, and demonstrate a complete understanding of this phase model is imperative to understanding the full polar lattice $(\ref{PolarLatticePhase})$.

\subsection{Outline of the Paper}

This paper is organized as follows. In Section~\ref{sec:RotWave} we precisely define what a rotating wave solution to equation (\ref{ComplexLattice}) is, and then in Section $\ref{sec:Model}$ we will properly describe how to obtain rotating wave solutions to $(\ref{ComplexLattice})$ in the anti-continuum limit. In particular, we will introduce the system of coupled phase equations which form the central equations of interest in this work. Section $\ref{sec:FiniteLattice}$ gives an overview of known results for rotating waves to coupled phase equations on finite square lattices. Then Section $\ref{sec:InfiniteLattice}$ carefully extends the results of the finite square lattice to a rotating wave solution to our system of coupled phase equations over the infinite lattice $\mathbb{Z}^2$. The results of Section $\ref{sec:InfiniteLattice}$ then give the existence of rotating wave solutions to $(\ref{ComplexLattice})$ in the anti-continuum limit. Section $\ref{sec:Persistence}$ is broken down into two subsections. The first subsection aims to formulate a conjecture into the persistence of rotating wave solutions to $(\ref{ComplexLattice})$ away from the anti-continuum limit, whereas the second subsection of Section $\ref{sec:Persistence}$ discusses the many technical difficulties in formally extending the existence work of this paper out of the anti-continuum limit. This is all done in the service that the reader fully comprehends just how nontrivial of a task the persistence results become when dealing with the infinite lattice setting of this work.

\section{Rotating Wave Solutions in Lattice Systems} \label{sec:RotWave} 

The central question of this work is to obtain rotating wave solutions to the system $(\ref{ComplexLattice})$. In the continuous spatial setting of partial differential equations, rotating waves are defined so that their temporal evolution is equivalent to a spatial rotation. To properly define a rotating wave solution in this discrete spatial context we will make use of the rotation operator acting on the indices of the lattice given by
\begin{equation} \label{RotationOperator}
	R(z_{i,j}) = z_{j,1-i},
\end{equation}  
where we rotate the lattice clockwise through an angle of $\pi/2$ about a theoretical centre cell at $i=j=1/2$. This theoretical centre will act as the centre of rotation for our rotating wave solution, although due to the translational invariance of the lattice this centre can be chosen to lie between any square arrangement of cells and still give a rotating wave solution. The effect this operator has on the closest cells to its centre of rotation is shown in Figure $\ref{fig:Rotation}$. For rotations through the angle $\pi$ we merely apply $R$ to itself, denoted $R^2$. Similarly, for clockwise rotations through the angle $3\pi/2$ (or counterclockwise rotations through the angle $\pi/2$) we apply $R$ to itself three times, denoted $R^3$. This rotation operator works only at the lattice level, and therefore does not alter the internal dynamics of the individual cells. That is, we merely move cells around in the lattice with this operator, but never alter their time-dependent dynamics.   

\begin{figure} 
	\centering
	\includegraphics[height = 5cm]{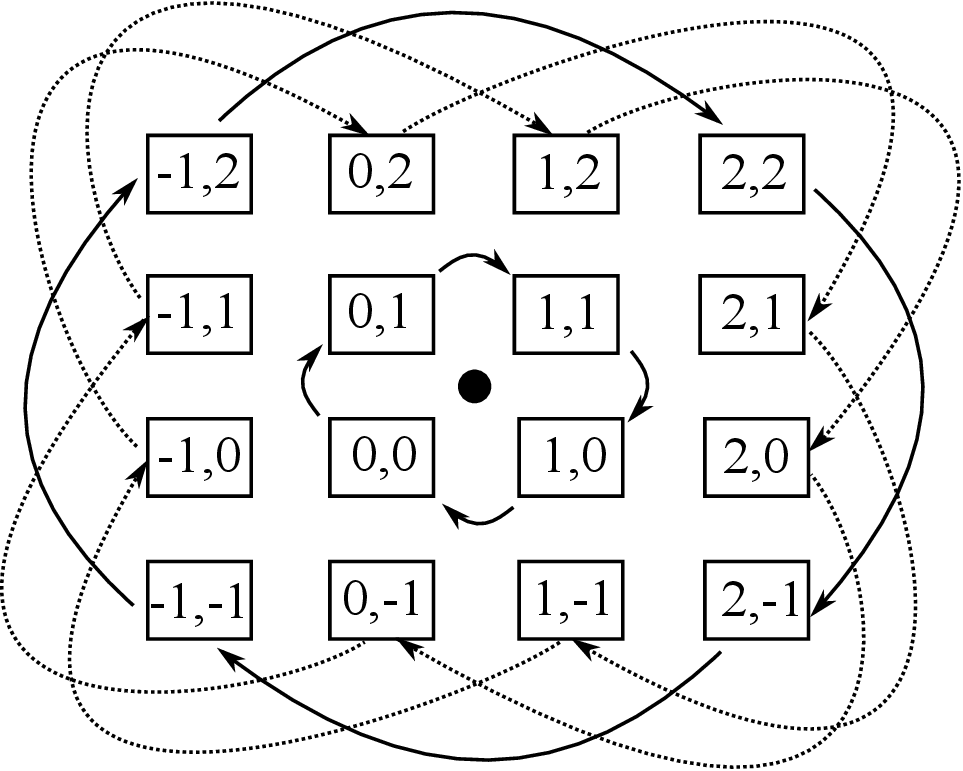}
	\caption{A diagram showing how the rotation operator defined in ($\ref{RotationOperator}$) effects elements of the lattice. The operator rotates lattice points by $\pi/2$ about a theoretical centre cell at $i=j=1/2$, represented by the dot in the centre of the diagram.}
	\label{fig:Rotation}
\end{figure} 

Having now defined an analogous spatial rotation in the discrete spatial setting, we turn to the central definition of this work, which details that rotating elements of the lattice about the centre of rotation simply leads to a phase advance of a quarter period. That is, as in the continuous spatial setting we seek a solution for which spatial rotation is equivalent to temporal translation. We will work to find a time-periodic solution defined by this rotational symmetry. 

\begin{defn} \label{def:RotWaveLDS}
	A {\bf rotating wave solution} of system $(\ref{ComplexLattice})$, denoted $\{z_{i,j}(t)\}_{(i,j)\in\mathbb{Z}^2}$, is a periodic solution with period $T > 0$ such that for all $(i,j) \in \mathbb{Z}^2$ and $t \in \mathbb{R}$ we have
	\begin{equation} \label{RotatingSymmetry}
		R(z_{i,j}(t)) = z_{i,j}(t + T/4).
	\end{equation}
\end{defn}

With the ansatz $z_{i,j}(t) = r_{i,j}(t)e^{i(\omega t + \theta_{i,j}(t))}$ introduced in (\ref{ComplexPolarForm}), we can see that obtaining steady-state solutions to (\ref{PolarLattice}), denoted $\{(\bar{r}_{i,j},\bar{\theta}_{i,j})\}_{(i,j)\in\mathbb{Z}^2}$, satsifying 
\begin{equation}
	\begin{split}
		R(r_{i,j}(t)) &= r_{i,j}(t), \\
		R(\theta_{i,j}(t)) &= \theta_{i,j}(t) + \frac{\pi}{2},
	\end{split}
\end{equation}
for all $t \geq 0$ and $(i,j) \in \mathbb{Z}^2$, leads to a rotating wave solution to (\ref{ComplexLattice}). Indeed, we can see that 
\begin{equation}
	\begin{split}
	R(z_{i,j}(t)) &= R(\bar{r}_{i,j})e^{i(\omega t + R(\bar{\theta}_{i,j})} \\ 
	&= \bar{r}_{i,j}e^{i(\omega t + \bar{\theta}_{i,j} + \frac{\pi}{2})} \\ 
	&= \bar{r}_{i,j}e^{i(\omega (t + \frac{\pi}{2\omega}) + \bar{\theta}_{i,j}} \\ 
	&= z_{i,j}\bigg(t + \frac{\pi}{2\omega}\bigg),  
	\end{split}
\end{equation}
for all $(i,j) \in \mathbb{Z}^2$, so that we have satisfied the Definition~\ref{def:RotWaveLDS} with period $T = \frac{2\pi}{\omega}$. In the following section we will detail exactly how we wish to obtain these steady-state solutions that lead to such a rotating wave solution to (\ref{ComplexLattice}).

\section{Solutions in the Anti-Continuum Limit} \label{sec:Model} 

We proceed with the polar decomposition $(\ref{PolarLattice})$ of $(\ref{ComplexLattice})$ detailed in the introduction. Our goal is to find time-periodic rotating wave solutions to $(\ref{ComplexLattice})$, and thus from our discussion in the previous section we reduce to searching for non-trivial steady-state solutions of the polar variables equations $(\ref{PolarLattice})$. That is, we are required to solve the infinite system of nonlinear equations given by
\begin{equation} \label{0Eqns_old}
	\begin{split}
	\begin{aligned}
		&0 = \alpha\sum_{i',j'} (r_{i',j'}\cos(\theta_{i',j'} - \theta_{i,j}) - r_{i,j}) + r_{i,j}(1 - r_{i,j}^2), \\
		&0 =  \alpha\sum_{i',j'} \frac{r_{i',j'}}{r_{i,j}}\sin(\theta_{i',j'} - \theta_{i,j}),
	\end{aligned}
	\end{split}	
\end{equation} 
for $\alpha \geq 0$ and $(i,j) \in\mathbb{Z}^2$. As mentioned in the introduction, our interest lies in the anti-continuum limit $\alpha \to 0^+$. Simply evaluating (\ref{0Eqns_old}) at $\alpha = 0$ will of course trivially solve the phase equations, but this gives no indication of what the solutions should look like for sufficiently small $\alpha > 0$, which is the goal of this work. Therefore, solutions of (\ref{0Eqns_old}) in the anti-continuum limit should really be interpreted as the limit of solutions with small $\alpha > 0$ as $\alpha \to 0^+$, as opposed to solutions of (\ref{0Eqns_old}) evaluated at $\alpha = 0$. Hence, solving systems in the anti-continuum limit bears a striking resemblance to obtaining and understanding the flow on slow manifolds in fast-slow dynamical system.     

Since $\alpha$ appears only as a a multiplicative constant in the second system of equations, searching for nontrivial solutions with $\alpha \geq 0$ to the polar variables equations $(\ref{PolarLattice})$ requires solving  
\begin{equation} \label{0Eqns}
	\begin{split}
	\begin{aligned}
		&0 = \alpha\sum_{i',j'} (r_{i',j'}\cos(\theta_{i',j'} - \theta_{i,j}) - r_{i,j}) + r_{i,j}(1 - r_{i,j}^2), \\
		&0 =  \sum_{i',j'} \frac{r_{i',j'}}{r_{i,j}}\sin(\theta_{i',j'} - \theta_{i,j}), 
	\end{aligned}
	\end{split}	
\end{equation} 
for each $(i,j) \in \mathbb{Z}^2$. As previously noted, any steady-state solutions $\{(\bar{r}_{i,j},\bar{\theta}_{i,j})\}_{(i,j)\in\mathbb{Z}^2}$ leads to solutions of the Ginzburg-Landau system $(\ref{ComplexLattice})$ of the form
\begin{equation} \label{zSoln}
	z_{i,j}(t) = \bar{r}_{i,j}e^{{\rm i}(\omega t + \bar{\theta}_{i,j})},
\end{equation}
for all $(i,j) \in \mathbb{Z}^2$. That is, each $z_{i,j}(t)$ is oscillating with a frequency of $2\pi/\omega$ but differ through time-independent phase-lags $\bar{\theta}_{i,j}$ and magnitudes $\bar{r}_{i,j}$.  

One can see that upon letting $\alpha \to 0^+$ in (\ref{0Eqns}) we have that the radial components not only decouple from their nearest-neighbours, but also from their associated phase variables. This leaves one to solve the simple polynomial equation
\begin{equation} \label{RadialRoots}
	r_{i,j}(1 - r_{i,j}^2) = 0.
\end{equation}
The only positive solution to $(\ref{RadialRoots})$ is given by $r_{i,j} = 1$ for all $(i,j)\in\mathbb{Z}^2$. This then reduces $(\ref{0Eqns})$ to solving 
\begin{equation} \label{PolarLattice2}
	\sum_{i',j'} \sin(\theta_{i',j'} - \theta_{i,j}) = 0,
\end{equation}	 
for each $(i,j)$, which correspond to the aforementioned time-independent phase-lags for a potential solution of the form $(\ref{zSoln})$. More precisely, solutions to system $(\ref{PolarLattice2})$ provide a leading order expansion of solutions to the full Ginzburg-Landau system with $\alpha \to 0^+$ of the form 
\begin{equation} \label{ACSoln}
	z_{i,j}(t) = e^{{\rm i}(\omega t + \bar{\theta}_{i,j})}. 	
\end{equation} 
From our work in the previous section, we can ensure that (\ref{ACSoln}) is a rotating wave by obtaining solutions $\{\bar{\theta}_{i,j}\}_{(i,j)\in\mathbb{Z}}$ to system (\ref{PolarLattice2}) such that 
\begin{equation}
	R(\bar{\theta}_{i,j}) = \bar{\theta}_{i,j} + \frac{\pi}{2}.
\end{equation}

In this work we solve $(\ref{PolarLattice2})$ by considering more general coupling functions that retain the necessary characteristics of the sine function. Specifically we will work to solve equations of the form
\begin{equation} \label{HZeros}
	\sum_{i',j'} H(\theta_{i',j'} - \theta_{i,j}) = 0,
\end{equation}  
for all $(i,j)\in\mathbb{Z}^2$ for coupling functions $H:S^1 \to S^1$ satisfying the following hypothesis:

\begin{hyp} \label{Hyp1} 
The coupling function $H: S^1 \to S^1$ is such that
\begin{itemize}
	\item $H \in C^\infty(S^1)$,
	\item $H(x + 2\pi) = H(x)$ for all $x\in S^1$,
	\item $H(-x) = -H(x)$ for all $x\in S^1$,
	\item $H'(x) > 0$ for all $x \in (\frac{-\pi}{2},\frac{\pi}{2})$.
\end{itemize} 
\end{hyp}

\noindent The reader should further note that the final two conditions of Hypothesis $\ref{Hyp1}$ can be combined to see that we necessarily have $H(0) = 0$ and
\begin{equation} \label{Positivity}
	H(x) > 0,\ \ \ \ \ x\in(0,\frac{\pi}{2}].
\end{equation} 
In the coming sections we will see how this set of conditions is minimal in that each is necessary for the results obtained in this work.   

\begin{rmk} Throughout this manuscript we will simply refer to a solution $\{\bar{\theta}_{i,j}\}_{(i,j)\in\mathbb{Z}}$ of the system of equations (\ref{HZeros}) satisfying the symmetry condition
\begin{equation}
	R(\bar{\theta}_{i,j}) = \bar{\theta}_{i,j} + \frac{\pi}{2}
\end{equation}
as a {\bf rotating wave}. The reason for this is due to this correspondence with the polar form (\ref{ACSoln}) and the discussion above. This will allow for the consideration of only the phase equation (\ref{HZeros}) throughout the following sections, thus adding some clarity in conveying the results.
\end{rmk}

As a brief aside, it should be noted that solving systems of type $(\ref{HZeros})$ is not just relevant to our study in this work, but also can be related to the study of identically coupled oscillators as well. That is, consider the system of coupled phase equations of the form
\begin{equation} \label{PhaseLDS}
	\dot{\theta}_{i,j}(t) = \omega + \sum_{i',j'} H(\theta_{i',j'}(t) - \theta_{i,j}(t)),  \ \ \ \ \ \ \ (i,j) \in\mathbb{Z}^2,
\end{equation} 
where $\theta_{i,j}:\mathbb{R}^+ \to S^1$ for each $(i,j)\in\mathbb{Z}^2$. By introducing the ansatz $\theta_{i,j}(t) = \omega t + \bar{\theta}_{i,j}$, where each oscillator has the same frequency but differs through the time-independent phase-lag $\bar{\theta}_{i,j}$, we reduce $(\ref{PhaseLDS})$ to solving 
\begin{equation}
	\sum_{i',j'} H(\bar{\theta}_{i',j'} - \bar{\theta}_{i,j}) = 0,
\end{equation} 
an equivalent system to $(\ref{HZeros})$ above. The coupled phase model $(\ref{PhaseLDS})$ can come as a generalization of the celebrated Kuramoto model which has widespread applications, particularly in neuroscience $\cite{Cumin, Kuramoto}$. There has been an extensive body of work on one-dimensional lattices, or chains, of coupled systems of phase equations with similar coupling functions in both the finite and infinite settings $\cite{ErmentroutKopell,ErmentroutKopell2}$. The study of two-dimensional lattices remains mostly unexplored, with the exception of some work on the finite square lattice $\cite{ErmentroutRen,ErmentroutSpiral}$.

Finally, one should note that the lattice structure and nearest-neighbour connections give systems $(\ref{ComplexLattice})$ and $(\ref{HZeros})$ a natural underlying graph theoretic geometry which we exploit throughout this work. We recall that nearest-neighbours are one step along the lattice from each other and then extend this notion inductively so that we say an element is $k$ steps away from another element if the shortest path along the lattice via nearest-neighbour connections requires us to move through exactly $k$ lattice points. For example, the cells with indices $(1,0)$ and $(2,2)$ are said to be 3 steps away from each other with an example of a shortest path between these indices via nearest-neighbour connections given by $(1,0)\to(2,0)\to(2,1)\to(2,2)$. Aside from this geometric perspective, this can be quantified analytically by saying that two elements indexed by the lattice points $(i_1,j_1)$ and $(i_2,j_2)$ are $k$ steps away from each other if 
\begin{equation}
	|i_1 - i_2| + |j_1 - j_2| = k.
\end{equation}  

This notion of distance along the lattice structure has important implications for systems of type $(\ref{PhaseLDS})$, which are used in the following section in order to solve $(\ref{HZeros})$ on a finite lattice. We see that the first derivative of any oscillator depends on the value of all nearest-neighbours as well as itself. It follows that the second derivative of any oscillator will depend on the derivative of each of its nearest-neighbours and the derivative of itself, implying that by the form of our LDS, we have the second derivative of any oscillator depending on the values of all oscillators two steps or less from it. More generally, the $k$th derivative of any oscillator depends on the value of all oscillators $k$ steps or less from it. This interconnectivity between oscillators will become crucial throughout this work and provide the basis for much of the work carried out.

\section{Rotating Waves on Finite Lattices} \label{sec:FiniteLattice} 

It was shown by Ermentrout and Paullet in $\cite{ErmentroutSpiral}$ that there exists rotating wave solutions on finite square lattices. Here we will review this work and demonstrate how it can be extended to give solutions on the infinite lattice. We obtain solutions to the finite lattice version of $(\ref{HZeros})$ as steady-states to a dynamical system. Begin by fixing an integer $N \geq 2$ to consider the truncated phase equations
\begin{equation} \label{FinitePhase}
	\dot{\theta}_{i,j} = \sum_{i',j'} H(\theta_{i',j'} - \theta_{i,j}),
\end{equation}   
where $1-N \leq i,j \leq N$ and the sum is over the nearest-neighbours of the cell $(i,j)$ on the finite square integer lattice with side lengths $2N$. Note that not all cells have four neighbours in this case because of the truncation to a finite lattice. That is, those cells along the edges have at most three nearest-neighbours included in the sum, and if they are a corner then there will be only two nearest-neighbours in the sum. 

We will follow the arguments laid out in $\cite{ErmentroutSpiral}$ to show that there exists a steady-state solution to $(\ref{FinitePhase})$ satisfying the definition of a rotating wave given above. Hence, we seek a steady-state solution with the symmetry shown in Figure $\ref{fig:PhaseSolution}$, which satisfies the condition $R(\bar{\theta}_{i,j}) = \bar{\theta}_{i,j} + \frac{\pi}{2}$ plus an additional symmetry (due to the fact that $H$ is considered to be an odd function) which will be discussed shortly. Ermentrout and Paullet detail that we may therefore reduce the number of equations from $4N^2$ to $\frac{1}{2}N(N-1)$ by focusing on those whose indices belong to the set
\[
	\Lambda_N = \{(i,j): 1 \leq j < i \leq N,\}
\] 
which we refer to as the {\em reduced system}. The reduced system is represented by the shaded cells in Figure $\ref{fig:PhaseSolution}$. 

\begin{figure} 
	\centering
	\includegraphics[width = 9cm]{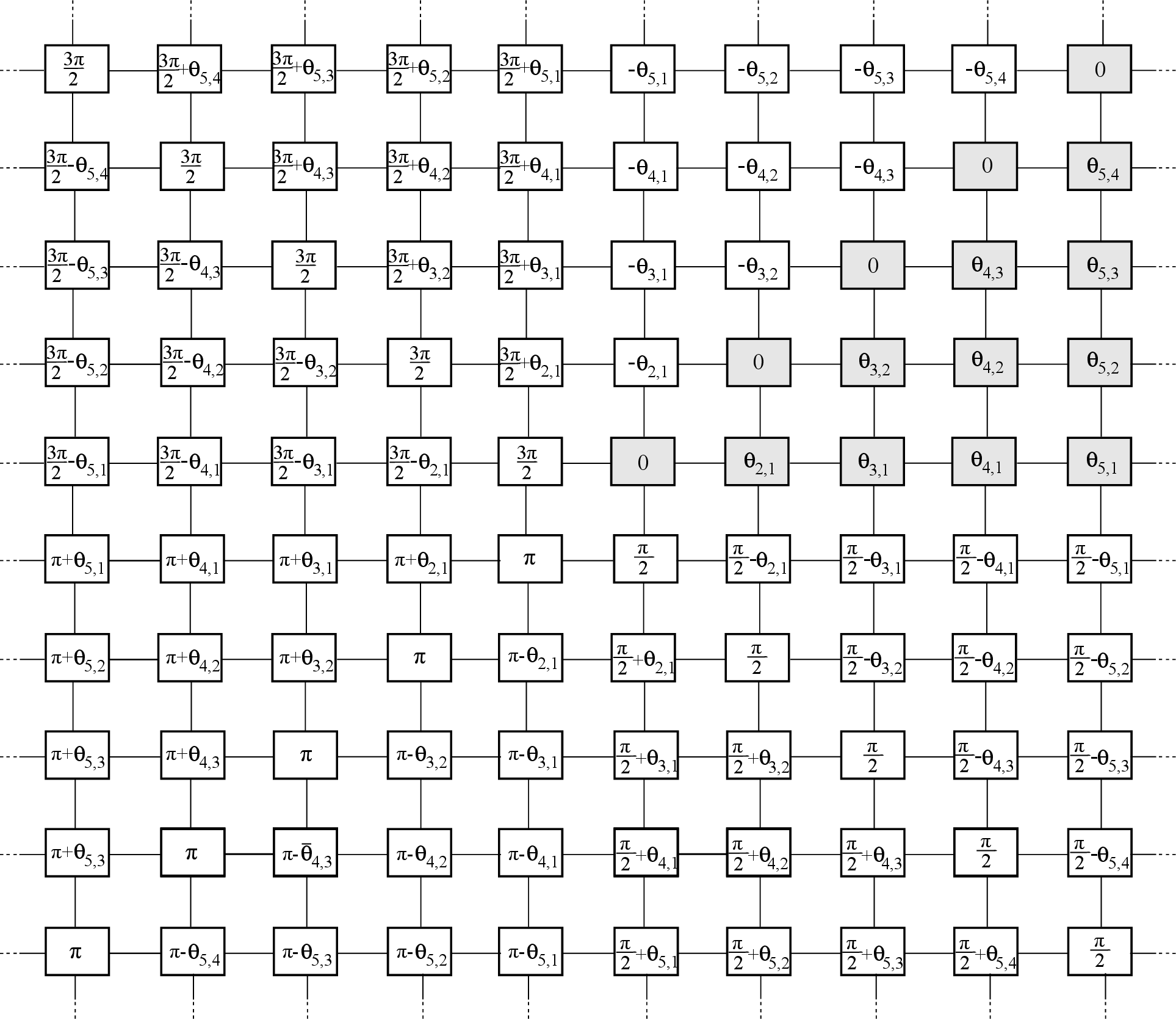}
	\caption{Symmetry of the phase-locked solution on the finite lattice. The shaded cells represent the reduced system.}
	\label{fig:PhaseSolution}
\end{figure}  

Upon applying these reductions to the system we arrive at the problem of finding a steady-state to the system of equations given by
\begin{equation} \label{ReducedFinitePhase}
	\dot{\theta}_{i,j} = \sum_{i',j'} H(\theta_{i',j'} - \theta_{i,j}),\ \ \ \ \ (i,j)\in \Lambda_N.	
\end{equation}
with initial conditions $\theta_{i,j}(0) = \frac{\pi}{4}$, $(i,j)\in \Lambda_N$. We also impose the boundary conditions $\theta_{i,i} = 0$ and $\theta_{i,0} = \frac{\pi}{2} - \theta_{i,1}$ for $1 \leq i \leq N$ which are reflected in Figure $\ref{fig:PhaseSolution}$ for reference. These boundary conditions lead to certain properties which are used to show the existence of a steady-state solution. In particular, the cells directly below the diagonal ($j = i-1$) are connected to two diagonal elements (above and to the left) which reduces their differential equations to
\begin{equation} \label{SpecialForm1}
	\begin{split}
	\begin{aligned}
		\dot{\theta}_{i,i-1} &= H(0 - \theta_{i,i-1}) + H(0 - \theta_{i,i-1}) + H(\theta_{i+1,i-1} - \theta_{i,i-1}) + H(\theta_{i,i-2} - \theta_{i,i-1}) \\
		&= -2H(\theta_{i,i-1}) + H(\theta_{i+1,i-1} - \theta_{i,i-1}) + H(\theta_{i,i-2} - \theta_{i,i-1}).
	\end{aligned}
	\end{split}
\end{equation}  
Here we have used the odd symmetry of the coupling function. Also, the cells in the first row of the reduced system ($j=1$) have a special term due to their connection with the boundary terms at $j = 0$ given by
\begin{equation} \label{SpecialForm2}
	\begin{split}
	\begin{aligned}
		\dot{\theta}_{i,1} &= H\bigg(\frac{\pi}{2} - \theta_{i,1} - \theta_{i,1}\bigg) + H(\theta_{i,2} - \theta_{i,1}) + H(\theta_{i-1,1} - \theta_{i,1}) + H(\theta_{i+1,1} - \theta_{i,1}) \\
		&= H\bigg(\frac{\pi}{2} -2\theta_{i,1}\bigg) + H(\theta_{i,2} - \theta_{i,1}) + H(\theta_{i-1,1} - \theta_{i,1}) + H(\theta_{i+1,1} - \theta_{i,1}).
	\end{aligned}
	\end{split}
\end{equation} 

\begin{rmk}
One will notice that there is a slight difference in the initial conditions between the work in $\cite{ErmentroutSpiral}$ and of that which is given here. In the former, the initial conditions are taken to be $\theta_{i,j}(0) = 0$, whereas here we will take them to be $\theta_{i,j}(0) = \frac{\pi}{4}$. This does not lead to any significantly different analysis, but will be useful in extending these results to the infinite lattice. 
\end{rmk}

The proof that a rotating wave solution exists on the finite lattice is broken down into two lemmas which together show that the trajectories are decreasing and bounded below for all $t > 0$. These lemmas together show that the trajectories therefore tend to an equilibrium as $t \to \infty$. The proof will only be briefly summarized here as it is nearly the same as that given by Ermentrout and Paullet, with the only distinction being that of a sign change due to the choice of initial conditions. Furthermore, an understanding of the methods employed in the proof on the finite lattice leads to a greater understanding of those used to extend to the infinite lattice, particularly in the proof of Lemma~\ref{lem:Horizontal} which closes out this section.   

\begin{lemma} \label{lem:ErmentroutLemma1} 
	For each fixed $N\geq 2$, given $\theta_{i,j}(0) = \frac{\pi}{4}$, there exists a $t_0 > 0$ such that $\dot{\theta}_{i,j}(t) < 0$ for all $0 < t < t_0$ and $(i,j) \in \Lambda_N$. 
\end{lemma}  

\begin{proof}
	Let us observe that with the initial conditions $\theta_{i,j}(0) = \frac{\pi}{4}$ 
	\begin{equation}
		\dot{\theta}_{i,j}(0) = \sum_{i',j'} H(0) = 0,
	\end{equation}
	for all $j \neq i-1$ because $H(0) = 0$ by the odd symmetry of $H$. In the case when $j = i -1$, from $(\ref{SpecialForm1})$ we get 
	\begin{equation}
		\dot{\theta}_{i,j}(0) = -2H\bigg(\frac{\pi}{4}\bigg) + 2H(0) < 0.
	\end{equation}
	This implies that there is an interval of $t$ values to the right of zero in which the value $\theta_{i,j}(t)$ indexed by an element directly below the diagonal of the reduced system is decreasing. 
	
	An inductive argument will show that when $j \neq i -1$ we have
	\begin{equation} \label{DerivativeCondition1}
		\left. \frac{d^k \theta_{i,j}}{d t^k}\right|_{t=0} = 0,\ \ \ \ \ k = 1,\dots ,i - j - 1
	\end{equation}   
	and
	\begin{equation} \label{DerivativeCondition2}
		\left. \frac{d^{(i - j)} \theta_{i,j}}{d t^{(i - j )}}\right|_{t = 0} < 0. 
	\end{equation} 
That is, the number of steps an element of the reduced system is from the diagonal determines which order derivative will be nonzero first. Then conditions $(\ref{DerivativeCondition1})$ and $(\ref{DerivativeCondition2})$ together imply that upon expanding each $\theta_{i,j}(t)$ as a Taylor series about $t = 0$ we get 
\begin{equation} \label{TaylorSeries}
	\theta_{i,j}(t) = \frac{\pi}{4} + \frac{a_{i,j}}{(i-j)!} t^{i - j} + \mathcal{O}( |t|^{i - j + 1}),
\end{equation} 
where $a_{i,j} < 0$ is used to denote the term $(\ref{DerivativeCondition2})$ for each $(i,j)$. Differentiating $(\ref{TaylorSeries})$ with respect to $t$ provides the Taylor series for $\dot{\theta}_{i,j}(t)$ about $t = 0$, given by
\begin{equation}
	\dot{\theta}_{i,j}(t) = \frac{a_{i,j}}{(i-j-1)!} t^{i - j - 1} + \mathcal{O}( |t|^{i - j}). 
\end{equation} 
Since $a_{i,j} < 0$, we have that $\dot{\theta}_{i,j}(t) < 0$ for sufficiently small $t > 0$ for each $1 \leq j < i \leq N$. But then since there are only finitely many elements in the reduced system, it follows that there exists a $t_0 > 0$ small enough so that $\dot{\theta}_{i,j}(t) < 0$ for all $0 < t < t_0$ and $1 \leq j < i \leq N$, giving the desired result.
\end{proof}

\begin{lemma} \label{lem:ErmentroutLemma2} 
	For all $t > 0$ and $(i,j)\in\Lambda_N$ we have $0 < \theta_{i,j}(t) < \frac{\pi}{4}$ and $\dot{\theta}_{i,j}(t) < 0$.
\end{lemma}

\begin{proof}
	Assume the contrary. That is, let $\hat{t} > 0$ be the first place where either $\theta_{\hat{i},\hat{j}}(\hat{t}) = 0$ or $\dot{\theta}_{\hat{i},\hat{j}}(\hat{t}) = 0$ for some index $(\hat{i},\hat{j})$. We will break this proof up into two case: (1) to show that there cannot be an index which satisfies $\theta_{\hat{i},\hat{j}}(\hat{t}) = 0$ and (2) that there cannot be an index such that $\dot{\theta}_{\hat{i},\hat{j}}(\hat{t}) = 0$ for all $t > 0$ and finite.

	{\em \underline{Case 1:}} Working with the first case we assume that $\theta_{\hat{i},\hat{j}}(\hat{t}) = 0$ for some $(\hat{i},\hat{j})$. Then we note that from the minimality of $\hat{t}$ we necessarily have $0 \leq \theta_{i,j}(\hat{t}) < \frac{\pi}{4}$ and $\dot{\theta}_{i,j}(\hat{t}) \leq 0$ for all $i,j$. Using the special form $(\ref{SpecialForm2})$ of those elements with $j = 1$ we see that $\hat{j} \neq 1$ since
	\begin{equation}
		\dot{\theta}_{\hat{i},1}(\hat{t}) = H\bigg(\frac{\pi}{2}\bigg) + H(\theta_{\hat{i},2}(\hat{t})) + H(\theta_{\hat{i}-1,1}(\hat{t})) + \left\{
     				\begin{array}{lr}
       				H(\theta_{\hat{i}+1,1}(\hat{t})) & : \hat{i} \neq N\\
       				0 & :  \hat{i} = N
     				\end{array}
  			 \right. ,
	\end{equation}  	
	which following $(\ref{Positivity})$ gives that each term in the sum is nonnegative, with $H(\frac{\pi}{2}) > 0$, thus giving that $\dot{\theta}_{\hat{i},1}(\hat{t})>0$, contradicting our assumption. 
	
	Moving to a pair $(\hat{i},\hat{j})$ with $\hat{j} \neq 1$ we see that if $\theta_{\hat{i},\hat{j}}(\hat{t}) = 0$ we get 
	\begin{equation}
		\dot{\theta}_{\hat{i},\hat{j}}(\hat{t}) = \sum_{i',j'}H(\theta_{\hat{i}',\hat{j}'}(\hat{t})).
	\end{equation}
	Since all elements of the reduced system are nonnegative at $t = \hat{t}$ we have that $\dot{\theta}_{\hat{i},\hat{j}}(\hat{t}) \geq 0$, again from $(\ref{Positivity})$. But from the definition of $\hat{t}$ we know that $\dot{\theta}_{\hat{i},\hat{j}}(\hat{t}) \leq 0$, implying that $\dot{\theta}_{\hat{i},\hat{j}}(\hat{t}) = 0$. The only way in which this is possible is if elements indexed by nearest-neighbours of $(\hat{i},\hat{j})$ are such that $\theta_{\hat{i}',\hat{j}'}(\hat{t}) = 0$ as well. This allows us to move to the nearest-neighbour indexed by $(\hat{i},\hat{j}-1)$ and perform the same analysis to similarly find that $\dot{\theta}_{\hat{i},\hat{j}-1}(\hat{t}) \geq 0$. Again, the only way in which this is possible is if all elements indexed by nearest-neighbours of $(\hat{i},\hat{j}-1)$ take the value $0$ at $t = \hat{t}$. This process continues by systematically moving down one cell at a time through the lattice until we find that $\theta_{\hat{i},2}(\hat{t}) = 0$. But then 
 	\begin{equation}
		\dot{\theta}_{\hat{i},2}(\hat{t}) =H(\theta_{\hat{i},1}(\hat{t})) + \cdots,
	\end{equation} 
where the neglected terms in the ellipsis are those elements to the left, above and to the right (if $\hat{i}\neq N$). As before, the neglected terms return a nonnegative value but one notices that since we have already shown that $\theta_{\hat{i},1}(\hat{t}) > 0$, it follows from $(\ref{Positivity})$ that $H(\theta_{\hat{i},1}(\hat{t})) > 0$. This then gives that $\dot{\theta}_{\hat{i},2}(\hat{t}) > 0$, which is a contradiction. Therefore, no $\theta_{i,j}(t)$ can reach $0$ at $t = \hat{t}$. 

{\em \underline{Case 2:}} Turning to the second case, we suppose $\dot{\theta}_{\hat{i},\hat{j}}(\hat{t}) = 0$. Clearly we cannot have $\dot{\theta}_{i,j}(\hat{t}) = 0$ for all $i,j$ since this would mean that we have reached an equilibrium point in finite $t$, which is impossible. Therefore, it can be assumed that there exists some index $(i_0,j_0)$ such that $\dot{\theta}_{i_0,j_0}(\hat{t}) < 0$ and again $0 \leq \theta_{i,j}(\hat{t}) < \frac{\pi}{4}$ and $\dot{\theta}_{i,j}(\hat{t}) \leq 0$ for all $i,j$. 

Let $\theta_{i_0,j_0}$ be the closest indexed element to $\theta_{\hat{i},\hat{j}}$ with $\dot{\theta}_{i_0,j_0}(\hat{t}) < 0$. Without loss of generality, we may assume that $(\hat{i},\hat{j})$ and $(i_0,j_0)$ are nearest-neighbours. Indeed, as previously noted, at $t = \hat{t}$ not all elements of the reduced system can have a derivative that vanishes, therefore there must be a pair of nearest-neighbours such that one of their derivatives vanishes at $t = \hat{t}$ and the other does not. If we can show that this cannot be possible, then necessarily we have that every element of the reduced system $\theta_{i,j}(t)$ is such that either $\dot{\theta}_{i,j}(\hat{t}) < 0$ or $\dot{\theta}_{i,j}(\hat{t}) = 0$. Since the later is impossible, we must have the former, completing the proof.   

Now, we proceed under the assumption that $(\hat{i},\hat{j})$ and $(i_0,j_0)$ are nearest-neighbours. The assumption $\dot{\theta}_{\hat{i},\hat{j}}(\hat{t}) = 0$ implies that
\begin{equation}
	\ddot{\theta}_{\hat{i},\hat{j}}(\hat{t}) = \sum_{\hat{i}',\hat{j}'} H'(\theta_{\hat{i}',\hat{j}'}(\hat{t}) - \theta_{\hat{i},\hat{j}}(\hat{t}))\dot{\theta}_{\hat{i}',\hat{j}'}(\hat{t}),
\end{equation}
where $H'$ denotes the derivative of $H$ with respect to its argument. By assumption, at $t = \hat{t}$ we have $0 \leq \theta_{i,j}(\hat{t}) < \frac{\pi}{4}$ for all $1 \leq j < i \leq N$, so that
\begin{equation}
	|\theta_{\hat{i}',\hat{j}'}(\hat{t}) - \theta_{\hat{i},\hat{j}}(\hat{t})| < \frac{\pi}{4}.
\end{equation}   
Since $H$ is assumed to be strictly increasing on $(\frac{-\pi}{2},\frac{\pi}{2})$, it follows that 
\begin{equation}
	H'(\theta_{\hat{i}',\hat{j}'}(\hat{t}) - \theta_{\hat{i},\hat{j}}(\hat{t})) > 0.	
\end{equation}
Furthermore, every element of the reduced system satisfies $\dot{\theta}_{i,j}(\hat{t}) \leq 0$, with the additional assumption that $\dot{\theta}_{i_0,j_0}(\hat{t}) < 0$ giving 
\begin{equation}
	\ddot{\theta}_{\hat{i},\hat{j}}(\hat{t}) = \sum_{\hat{i}',\hat{j}'} H'(\theta_{\hat{i}',\hat{j}'}(\hat{t}) - \theta_{\hat{i},\hat{j}}(\hat{t}))\dot{\theta}_{\hat{i}',\hat{j}'}(\hat{t}) < 0,	
\end{equation}  
since $(i_0,j_0)$ and $(\hat{i},\hat{j})$ are nearest-neighbours. 

Now, expanding $\dot{\theta}_{\hat{i},\hat{j}}(t)$ as a Taylor series about $t = \hat{t}$ gives 
\begin{equation}
	\dot{\theta}_{\hat{i},\hat{j}}(t) = \ddot{\theta}_{\hat{i},\hat{j}}(\hat{t})(t - \hat{t}) + \mathcal{O}(|t - \hat{t}|^2),	
\end{equation}
since $\dot{\theta}_{\hat{i},\hat{j}}(\hat{t}) = 0$. Since $\ddot{\theta}_{\hat{i},\hat{j}}(\hat{t}) < 0$, we have that there exists a sufficiently small nontrivial interval of $t$ values to the left of $\hat{t}$ such that $\dot{\theta}_{\hat{i},\hat{j}}(t) > 0$. But this implies that there exists a positive $t' < \hat{t}$ such that $\dot{\theta}_{\hat{i},\hat{j}}(t') = 0$ since Lemma $\ref{lem:ErmentroutLemma1}$ gave that $\dot{\theta}_{\hat{i},\hat{j}}(t) < 0$ for sufficiently small $t > 0$. This therefore contradicts the minimality of $\hat{t}$ and hence from our arguments above, no derivative of the $\theta_{i,j}(t)$ can vanish at a finite value of $t$, completing the proof.
\end{proof}

In summary, it was shown that each element of the reduced system is such that
\begin{equation}
	0 < \theta_{i,j}(t) < \frac{\pi}{4}\ \ \ \ \  {\rm and}\ \ \ \ \  \dot{\theta}_{i,j}(t) < 0	
\end{equation}  
for all $t > 0$ and $(i,j)\in\Lambda_N$. Hence, $\theta_{i,j}(t)$ is a decreasing function which is bounded below, so
\begin{equation}
	\bar{\theta}_{i,j} = \lim_{t \to \infty} \theta_{i,j}(t)
\end{equation}
exists and lies in the interval $[0,\frac{\pi}{4})$. Therefore $\{\bar{\theta}_{i,j}\}_{1\leq j < i \leq N}$ gives an equilibrium for the reduced system. As a brief aside, one can further show via the same methods used in Lemma $\ref{lem:ErmentroutLemma2}$ that $\bar{\theta}_{i,j} > 0$ for all $1\leq j < i \leq N$.

\begin{figure} 
	\centering
	\includegraphics[width = 7cm]{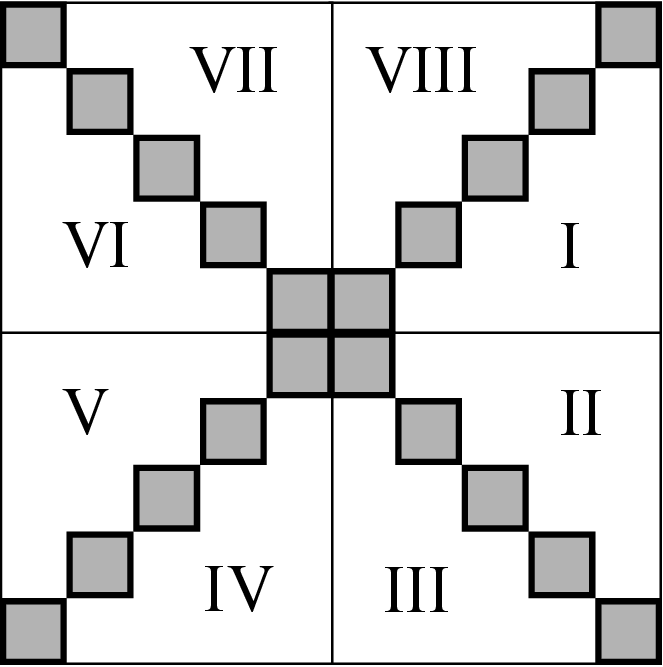}
	\caption{The eight distinct regions of the finite lattice defined by the reduced system.}
	\label{fig:ErmentroutDiagram}
\end{figure}  

To extend the solution of the reduced system to the entire square lattice we refer to Figure $\ref{fig:ErmentroutDiagram}$ where the finite $2N\times 2N$ lattice has been partitioned into eight distinct regions. To begin, it has already been remarked that those elements along the diagonal between regions $I$ and $VIII$ are fixed at $0$. Then those on the diagonal between regions $II$ and $III$ are fixed at $\frac{\pi}{2}$, those between regions $IV$ and $V$ are fixed at $\pi$ and those between regions $VI$ and $VII$ are fixed at $\frac{3\pi}{2}$. If we write $\bar{\theta}$ to be the solution of the reduced system found above, the solutions in each of the regions of Figure $\ref{fig:ErmentroutDiagram}$ are as follows:
\begin{equation} \label{SymmetryExtensions}
	\begin{split}
	\begin{aligned}
		I:&\ \bar{\theta} \to \bar{\theta} \\
		II:&\ \bar{\theta} \to \frac{\pi}{2} - \bar{\theta} \\
		III:&\ \bar{\theta} \to \frac{\pi}{2} + \bar{\theta} \\
		IV:&\ \bar{\theta} \to \pi - \bar{\theta} \\
		V:&\ \bar{\theta} \to \pi + \bar{\theta} \\
		VI:&\ \bar{\theta} \to \frac{3\pi}{2} - \bar{\theta} \\
		VII:&\ \bar{\theta} \to \frac{3\pi}{2} + \bar{\theta} \\
		VIII:&\ \bar{\theta} \to - \bar{\theta}.
	\end{aligned}
	\end{split}
\end{equation} 
Note that these extensions give exactly the symmetry of the solution shown in Figure $\ref{fig:PhaseSolution}$, and satisfy the condition
\begin{equation}
	R(\bar{\theta}_{i,j}) = \bar{\theta}_{i,j} + \frac{\pi}{2}
\end{equation}
giving a rotating wave solution to (\ref{FinitePhase}).

Finally, we extend the previous results slightly with the following lemma, which will be crucial for our extension to the infinite lattice in the following section. 

\begin{lemma} \label{lem:Horizontal} 
	Let $N \geq 2$ and finite. If we denote $\bar{\theta}_{i,j}$ as the solutions on the finite $2N\times 2N$ lattice in the reduced system $(\ref{ReducedFinitePhase})$, then $\bar{\theta}_{i+1,j} \geq \bar{\theta}_{i,j}$ for all $1 \leq j \leq i \leq N-1$.   	
\end{lemma}

\begin{proof}
	This proof follows in a very similar way to how we proceeded in Lemmas $\ref{lem:ErmentroutLemma1}$ and $\ref{lem:ErmentroutLemma2}$ to show that the $\theta_{i,j}(t)$ of the reduced system are decreasing for all $t > 0$. We begin by using the conditions of the derivatives $(\ref{DerivativeCondition1})$ and $(\ref{DerivativeCondition2})$ to find that there exists a small interval to the right of zero for which $\theta_{i,j}(t) < \theta_{i+1,j}(t)$ for all $t > 0$ belonging to this interval. Then we assume that this interval is finite to arrive at a contradiction showing that these inequalities hold for all $t > 0$. Upon showing that these inequalities hold for all $t > 0$, we may extend them to the steady-state solution by having $t \to \infty$, giving the desired result.  

	Recall that we take our initial conditions to be $\theta_{i,j}(0) = \frac{\pi}{4}$ for all $(i,j)$ in the reduced system and those on the diagonal fixed at $0$. From Lemma $\ref{lem:ErmentroutLemma1}$ we have $\dot{\theta}_{i,i-1}(0) < 0$ and 
	\begin{equation}
		\left. \frac{d^k \theta_{i,j}}{d t^k}\right|_{t=0} = 0,\ \ \ \ \ k = 1,\dots ,i - j - 1
	\end{equation}   
	with
	\begin{equation}
		\left. \frac{d^{(i - j)} \theta_{i,j}}{d t^{(i - j )}}\right|_{t = 0} < 0,
	\end{equation} 
	for each $j \neq i -1$. Then from these facts we have that upon expanding the difference $\theta_{i,j}(t) - \theta_{i+1,j}(t)$ as a Taylor series about $t = 0$ the first $(i - j - 1)$ terms vanish leaving
	\begin{equation}
		\theta_{i,j}(t) - \theta_{i+1,j}(t) = \frac{b_{i,j}}{(i-j)!} t^{(i - j)} + \mathcal{O}(|t|^{i-j+1}), 	
	\end{equation}
	where
	\begin{equation}
		b_{i,j} = \left. \frac{d^{(i - j)} \theta_{i,j}}{d t^{(i - j )}}\right|_{t = 0} - \underbrace{\left. \frac{d^{(i - j)} \theta_{i+1,j}}{d t^{(i - j )}}\right|_{t = 0}}_{= 0} = \left. \frac{d^{(i - j)} \theta_{i,j}}{d t^{(i - j )}}\right|_{t = 0} < 0.
	\end{equation}
	Therefore, there exists small $t > 0$ such that $\theta_{i,j}(t) - \theta_{i+1,j}(t) < 0$, thus implying that $\theta_{i,j}(t) < \theta_{i+1,j}(t)$ on this interval. Since this is true for all elements of the reduced system, which is finite, there exists a $t_0 > 0$ such that $\theta_{i,j}(t) < \theta_{i+1,j}(t)$ for all $t \in (0, t_0)$ and $1 \leq j < i \leq N - 1$. 
	
	We now assume that this ordering of the elements of the reduced system only persists for finite $t$. That is, let $t_0 > 0$ be the first value of $t$ in which the inequality no longer holds for all elements of the reduced system. Then there exists at least one index of the reduced system, $(i_0,j_0)$, such that $\theta_{i_0+1,j_0}(t_0) = \theta_{i_0,j_0}(t_0)$, and $\theta_{i+1,j}(t_0) \geq \theta_{i,j}(t_0)$ for all $(i,j) \neq (i_0,j_0)$. We first note that $j_0 \neq i_0$. Indeed, by Lemma $\ref{lem:ErmentroutLemma2}$ the elements of the reduced system satisfy $0 < \theta_{i,j}(t) < \frac{\pi}{4}$ for all $t > 0$ and one sees that 
	\begin{equation}
		\theta_{i_0+1,i_0}(t) > 0 = \theta_{i_0,i_0}(t).
	\end{equation}
	Hence, along every row of the reduced system, there must be at least one strict inequality at $t = t_0$. Therefore, without loss of generality, we may assume that the index $(i_0,j_0)$ is such that $\theta_{i_0+1,j_0}(t_0) = \theta_{i_0,j_0}(t_0) > \theta_{i_0-1,j_0}(t_0)$. 
	
	\begin{figure} 
		\centering
		\includegraphics[width = 6cm]{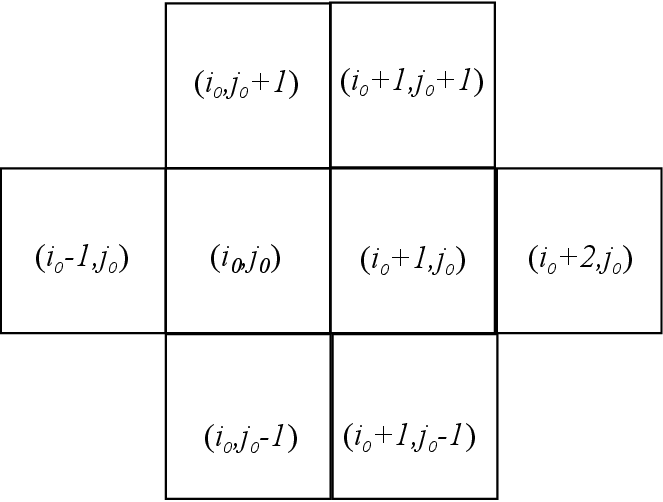}
		\caption{The lattice points indexed by $(i_0,j_0)$ and $(i_0+1,j_0)$ along with their nearest-neighbours.}
		\label{fig:NNHorizontal}
	\end{figure} 
	
	Now let us investigate $\dot{\theta}_{i_0+1,j_0}(t_0) - \dot{\theta}_{i_0,j_0}(t_0)$. Using the form of the differential equations given in $(\ref{FinitePhase})$ we see that 
	\begin{equation}
		\begin{aligned}
		\begin{split}
			\dot{\theta}_{i_0+1,j_0}&(t_0) - \dot{\theta}_{i_0,j_0}(t_0) = H(\theta_{i_0+1,j_0 + 1}(t_0) - \theta_{i_0,j_0}(t_0)) - H(\theta_{i_0,j_0 + 1}(t_0) - \theta_{i_0,j_0}(t_0))\\
			 &+ H(\theta_{i_0+1,j_0-1}(t_0) - \theta_{i_0,j_0}(t_0)) - H(\theta_{i_0,j_0 -1}(t_0) - \theta_{i_0,j_0}(t_0))\\
			 &+ \left\{
     				\begin{array}{lr}
       				H(\theta_{i_0+2,j_0}(t_0) - \theta_{i_0,j_0}(t_0)) - H(\theta_{i_0-1,j_0 }(t_0) - \theta_{i_0,j_0}(t_0)) & : i_0 \neq N-1\\
       				- H(\theta_{i_0-1,j_0 +1}(t_0) - \theta_{i_0,j_0}(t_0)) & :  i_0 = N-1
     				\end{array}
  			 \right.   	
		\end{split}
		\end{aligned}
	\end{equation} 
where $ \theta_{i_0+1,j_0}(t_0)$ has been replaced with $\theta_{i_0,j_0}(t_0)$ by our assumption. Furthermore, notice that the coupling between these cells does not appear in the difference since
\begin{equation}
	 H(\theta_{i_0+1,j_0}(t_0) - \theta_{i_0,j_0}(t_0)) = H(0) = 0,
\end{equation}
because $\theta_{i_0+1,j_0}(t_0) = \theta_{i_0,j_0}(t_0)$. Figure $\ref{fig:NNHorizontal}$ shows the location of these elements in relation to each other on the lattice for visual reference. Since $\theta_{i_0+1,j_0}(t_0) = \theta_{i_0,j_0}(t_0) > \theta_{i_0-1,j_0}(t_0)$ we see that upon using the odd symmetry of $H$ we have
	\begin{equation}
		 -H(\theta_{i_0-1,j_0}(t_0) - \theta_{i_0,j_0}(t_0)) = H(\theta_{i_0,j_0}(t_0) - \theta_{i_0-1,j_0}(t_0)) > 0.
	\end{equation}
	Moreover, in the case when $i_0 \neq N -1$ our assumption gives $\theta_{i_0+2,j_0}(t_0) \geq \theta_{i_0+1,j_0}(t_0) = \theta_{i_0,j_0}(t_0)$ which implies
	\begin{equation}
		H(\theta_{i_0+2,j_0}(t_0) - \theta_{i_0,j_0}(t_0)) \geq 0.	
	\end{equation}
	Again by the minimality of $t_0$ we get that $\theta_{i_0+1,j_0\pm 1}(t_0) \geq \theta_{i_0,j_0\pm 1}$, leading to the fact that 
	\begin{equation}
		H(\theta_{i_0+1,j_0\pm 1}(t_0) - \theta_{i_0,j_0}(t_0)) - H(\theta_{i_0,j_0\pm 1}(t_0) - \theta_{i_0,j_0}(t_0)) \geq 0.	
	\end{equation}
	Putting this all together reveals that 
	\begin{equation}
		\dot{\theta}_{i_0+1,j_0}(t_0) - \dot{\theta}_{i_0,j_0}(t_0) > 0.
	\end{equation} 
	Expanding $\theta_{i_0+1,j_0}(t) - \theta_{i_0,j_0}(t)$ as a Taylor series about $t = t_0$ gives
	\begin{equation}	
		\begin{aligned}
		\begin{split}
		\theta_{i_0+1,j_0}(t) - \theta_{i_0,j_0}(t) = & \\ 
		\underbrace{[\theta_{i_0+1,j_0}(t_0) - \theta_{i_0,j_0}(t_0)]}_{= 0} + &\underbrace{[\dot{\theta}_{i_0+1,j_0}(t_0) - \dot{\theta}_{i_0,j_0}(t_0)]}_{> 0}(t - t_0) + \mathcal{O}(|t - t_0|^2).  
		\end{split}
		\end{aligned}
	\end{equation}
	Thus, there exists an $\varepsilon > 0$ such that $\theta_{i_0+1,j_0}(t) - \theta_{i_0,j_0}(t) < 0$ on $(t_0 - \varepsilon, t_0)$. But it was already shown that for $t > 0$ sufficiently small we have $\theta_{i_0+1,j_0}(t) - \theta_{i_0,j_0}(t) > 0$, which from the Intermediate Value Theorem implies that there is some positive $t' < t_0$ such that $\theta_{i_0+1,j_0}(t') = \theta_{i_0,j_0}(t')$. This contradicts the minimality of $t_0$, thus giving that no such $t_0$ can exist.   
	
	Therefore, $\theta_{i+1,j}(t) > \theta_{i,j}(t)$ for all $t > 0$. Allowing $t \to \infty$ we see that the elements of the equilibrium must satisfy $\bar{\theta}_{i+1,j} \geq \bar{\theta}_{i,j}$ for all $1 \leq j \leq i \leq N-1$, giving the desired result. 
\end{proof}

\section{Rotating Waves on an Infinite Lattice} \label{sec:InfiniteLattice} 

We now demonstrate how the solutions on the finite lattice of the previous section can be extended to solutions to $(\ref{HZeros})$. Throughout this section when referring to the solution on a finite lattice, it should be understood as the specific solution found in the previous section, and illustrated in Figure $\ref{fig:PhaseSolution}$. We provide the following result.

\begin{thm} \label{thm:InfinitePhase} 
	The system of equations $(\ref{HZeros})$ exhibits a rotating wave solution.   
\end{thm}

The proof of Theorem $\ref{thm:InfinitePhase}$ is broken down into a series of lemmas for the ease of the reader. The following lemma is the main result used to prove Theorem~\ref{thm:InfinitePhase}. 
	
\begin{lemma} \label{lem:IncreasingWithLattice} 
	Let us denote $\bar{\theta}^{(N)}_{i,j}$ as the solutions of the finite $2N\times 2N$ lattice in the reduced system for any $N \geq 2$. Then $\bar{\theta}^{(N)}_{i,j} \leq \bar{\theta}^{(N+1)}_{i,j}$ for all $(i,j) \in \Lambda_N$. That is, the value of the equilibrium point at each index in the reduced system is increasing as a function of the size of the lattice.     
\end{lemma}

To prove Lemma~\ref{lem:IncreasingWithLattice}, we will fix $N \geq 2$ and assume there exists an index $(i_0,j_0)$ of the reduced system $(\ref{ReducedFinitePhase})$ such that $\bar{\theta}^{(N)}_{i_0,j_0} > \bar{\theta}^{(N+1)}_{i_0,j_0}$. The proof is then carried out by examining the case when $(i_0,j_0)$ is an index from the last column (i.e. $i_0 = N$), deriving a contradiction and systematically decreasing the possible value of $i_0$ one step at a time to show that no such $(i_0,j_0)$ can exist. We show that the contradiction which is derived for the last column $(i_0 = N)$ is easily extended to derive a contradiction in the cases that $i_0 < N$. This will in turn exhaust all possibilities of indices $(i_0,j_0)$ in a finite number of steps, thus showing that $\bar{\theta}^{(N)}_{i,j} \leq \bar{\theta}^{(N+1)}_{i,j}$ for all $1 \leq j < i \leq N$. We begin with the following lemma detailing the case $i_0 = N$. 

\begin{lemma}[Case $i_0 = N$] \label{lem:i_0=N}
	Let us denote $\bar{\theta}^{(N)}_{i,j}$ as the solutions of the finite $2N\times 2N$ lattice in the reduced system for any $N \geq 2$. Then $\bar{\theta}^{(N)}_{N,j} \leq \bar{\theta}^{(N+1)}_{N,j}$ for all $1 \leq j < N$.
\end{lemma}

\begin{proof}	
	Let us fix $N \geq 2$ and assume that there exists an index $(N,j_0)$ of the reduced system such that $\bar{\theta}^{(N)}_{N,j_0} > \bar{\theta}^{(N+1)}_{N,j_0}$. The goal of this proof will be to derive a contradiction of this assumption, thus proving the lemma.
	
	To begin, consider the difference of the differential equations $\dot{\theta}^{(N+1)}_{N,j_0} - \dot{\theta}^{(N)}_{N,j_0}$ evaluated at the respective equilibrium points for that size of lattice. This gives
	\begin{equation} \label{Nj0Diff}
		\begin{aligned}
		\begin{split}
			0 = &H(\bar{\theta}^{(N+1)}_{N,j_0+1} - \bar{\theta}^{(N+1)}_{N,j_0}) - H(\bar{\theta}^{(N)}_{N,j_0+1} - \bar{\theta}^{(N)}_{N,j_0}) \\
			& + H(\bar{\theta}^{(N+1)}_{N,j_0-1} - \bar{\theta}^{(N+1)}_{N,j_0}) - H(\bar{\theta}^{(N)}_{N,j_0-1} - \bar{\theta}^{(N)}_{N,j_0}) \\
			& + H(\bar{\theta}^{(N+1)}_{N-1,j_0} - \bar{\theta}^{(N+1)}_{N,j_0}) - H(\bar{\theta}^{(N)}_{N-1,j_0} - \bar{\theta}^{(N)}_{N,j_0}) \\
			& + H(\bar{\theta}^{(N+1)}_{N+1,j_0} - \bar{\theta}^{(N+1)}_{N,j_0}), \\
		\end{split}
		\end{aligned}
	\end{equation} 
	where there is an odd number of terms since the $2N\times 2N$ lattice does not have a right input at the index $(N,j_0)$. From Lemma $\ref{lem:Horizontal}$ we have that $ \bar{\theta}^{(N+1)}_{N+1,j_0} - \bar{\theta}^{(N+1)}_{N,j_0} \geq 0$, which from $(\ref{Positivity})$ implies that
	\begin{equation} \label{NewPositivity}
		H(\bar{\theta}^{(N+1)}_{N+1,j_0} - \bar{\theta}^{(N+1)}_{N,j_0}) \geq 0.	
	\end{equation}
	 Then using $(\ref{NewPositivity})$ we may rearrange $(\ref{Nj0Diff})$ to find that 
	 \begin{equation} 
		\begin{aligned}
		\begin{split}
			0 \geq\ &H(\bar{\theta}^{(N+1)}_{N,j_0+1} - \bar{\theta}^{(N+1)}_{N,j_0}) - H(\bar{\theta}^{(N)}_{N,j_0+1} - \bar{\theta}^{(N)}_{N,j_0}) \\
			& + H(\bar{\theta}^{(N+1)}_{N,j_0-1} - \bar{\theta}^{(N+1)}_{N,j_0}) - H(\bar{\theta}^{(N)}_{N,j_0-1} - \bar{\theta}^{(N)}_{N,j_0}) \\
			& + H(\bar{\theta}^{(N+1)}_{N-1,j_0} - \bar{\theta}^{(N+1)}_{N,j_0}) - H(\bar{\theta}^{(N)}_{N-1,j_0} - \bar{\theta}^{(N)}_{N,j_0}). 
		\end{split}
		\end{aligned}
	\end{equation}
	 This in turn implies that at least one of the following must be true:
	 \begin{itemize}
	 	\item $H(\bar{\theta}^{(N+1)}_{N,j_0+1} - \bar{\theta}^{(N+1)}_{N,j_0}) - H(\bar{\theta}^{(N)}_{N,j_0+1} - \bar{\theta}^{(N)}_{N,j_0}) \leq 0$,
		\item $H(\bar{\theta}^{(N+1)}_{N,j_0-1} - \bar{\theta}^{(N+1)}_{N,j_0}) - H(\bar{\theta}^{(N)}_{N,j_0-1} - \bar{\theta}^{(N)}_{N,j_0}) \leq 0$,
		\item $H(\bar{\theta}^{(N+1)}_{N-1,j_0} - \bar{\theta}^{(N+1)}_{N,j_0}) - H(\bar{\theta}^{(N)}_{N-1,j_0} - \bar{\theta}^{(N)}_{N,j_0}) \leq 0$.
	 \end{itemize}
	 By definition of our coupling function $H$, we have that $H'(x) > 0$ for all $x \in (\frac{-\pi}{2},\frac{\pi}{2})$, and recalling from Lemma $\ref{lem:ErmentroutLemma2}$ that the maximal difference between any two elements of the reduced system is strictly bounded by $\pi/2$, we find that the above conditions reduce to having at least one of the following being true
	\begin{itemize} 
		\item $\bar{\theta}^{(N+1)}_{N,j_0+1} - \bar{\theta}^{(N+1)}_{N,j_0} \leq \bar{\theta}^{(N)}_{N,j_0+1} - \bar{\theta}^{(N)}_{N,j_0} \implies \bar{\theta}^{(N)}_{N,j_0+1} > \bar{\theta}^{(N+1)}_{N,j_0+1}$,
		\item $\bar{\theta}^{(N+1)}_{N,j_0-1} - \bar{\theta}^{(N+1)}_{N,j_0} \leq \bar{\theta}^{(N)}_{N,j_0-1} - \bar{\theta}^{(N)}_{N,j_0} \implies \bar{\theta}^{(N)}_{N,j_0-1} > \bar{\theta}^{(N+1)}_{N,j_0-1}$,
		\item $\bar{\theta}^{(N+1)}_{N-1,j_0} - \bar{\theta}^{(N+1)}_{N,j_0} \leq \bar{\theta}^{(N)}_{N-1,j_0} - \bar{\theta}^{(N)}_{N,j_0} \implies \bar{\theta}^{(N)}_{N-1,j_0} > \bar{\theta}^{(N+1)}_{N-1,j_0}$.  
	\end{itemize}
	 Here these conditions have been reduced by recalling that by assumption $\bar{\theta}^{(N)}_{i_0,j_0} > \bar{\theta}^{(N+1)}_{i_0,j_0}$. For simplicity we relabel the index $(N,j_0)$ as $\eta_1$ and let $\eta_2$ to be the nearest-neighbour of $\eta_1$ with the property that $\bar{\theta}^{(N)}_{\eta_{2}} > \bar{\theta}^{(N+1)}_{\eta_{2}}$. 
	
	We note that there are restrictions on the choice of $\eta_2$ in certain cases. That is, if $j_0 = 1$ then by the form of the solutions on the finite lattice we necessarily have
	\begin{equation}
		\bar{\theta}^{(N+1)}_{i_0,0} = \frac{\pi}{2} - \bar{\theta}^{(N+1)}_{i_0,1} > \frac{\pi}{2} - \bar{\theta}^{(N)}_{i_0,1} = \bar{\theta}^{(N)}_{i_0,0},
	\end{equation} 
	meaning that $\eta_2$ cannot be below $\eta_1$ when $j_0 = 1$. Furthermore, if $j_0 = i_0 - 1$ then $ \bar{\theta}^{(N+1)}_{i_0,i_0} = 0 =  \bar{\theta}^{(N)}_{i_0,i_0}$ showing that $\eta_2$ cannot be above or to the left of $\eta_1$ when $j_0 = i_0 - 1$. Since $N \geq 2$ we can always find an index $\eta_2$ with the prescribed properties in either situation.

	We now apply a similar argument to the difference of the differential equations
	\begin{equation}
	 	(\dot{\theta}^{(N+1)}_{\eta_{1}} + \dot{\theta}^{(N+1)}_{\eta_{2}}) - (\dot{\theta}^{(N)}_{\eta_{1}} + \dot{\theta}^{(N)}_{\eta_{2}}) 
	\end{equation} 
	evaluated at the respective equilibrium for that size of lattice. This gives
	\begin{equation} \label{Nj0Diff2}
		\begin{aligned}
		\begin{split}
		0 = &\sum_{\eta'_1} \bigg[H(\bar{\theta}^{(N+1)}_{\eta'_{1}} - \bar{\theta}^{(N+1)}_{\eta_{1}}) - H(\bar{\theta}^{(N)}_{\eta'_{1}} - \bar{\theta}^{(N)}_{\eta_{1}})\bigg] \\
			&+ \sum_{\eta'_2} \bigg[H(\bar{\theta}^{(N+1)}_{\eta'_{2}} - \bar{\theta}^{(N+1)}_{\eta_{2}}) - H(\bar{\theta}^{(N)}_{\eta'_{2}} - \bar{\theta}^{(N)}_{\eta_{2}})\bigg],
		\end{split}
		\end{aligned}
	\end{equation}
	where we have paired the elements by their index. This expression can be simplified slightly by recalling that $\eta_1$ and $\eta_2$ are nearest-neighbours in the lattice. Therefore, the terms
	\begin{equation}
		H(\bar{\theta}^{(N+1)}_{\eta_1} - \bar{\theta}^{(N+1)}_{\eta_2})
	\end{equation}
	and
	\begin{equation}
		H(\bar{\theta}^{(N+1)}_{\eta_2} - \bar{\theta}^{(N+1)}_{\eta_1})	
	\end{equation}
	both appear in this sum. Thus, using the odd symmetry of the coupling function, these terms eliminate themselves from the sum. Similarly, the terms $-H(\bar{\theta}^{(N)}_{\eta_1} - \bar{\theta}^{(N)}_{\eta_2})$ and $-H(\bar{\theta}^{(N)}_{\eta_2} - \bar{\theta}^{(N)}_{\eta_1})$ cancel each other in the sum by the odd symmetry of the coupling function $H$. 
	
	Then $(\ref{Nj0Diff2})$ again has the term $H(\bar{\theta}^{(N+1)}_{N+1,j_0} - \bar{\theta}^{(N+1)}_{N,j_0})$ being nonnegative, coming from the index $\eta_1$ again. Rearranging $(\ref{Nj0Diff2})$ as above then shows that at least one of
	\begin{equation}
		H(\bar{\theta}^{(N+1)}_{\eta'_{1}} - \bar{\theta}^{(N+1)}_{\eta_{1}}) - H(\bar{\theta}^{(N)}_{\eta'_{1}} - \bar{\theta}^{(N)}_{\eta_{1}}) \leq 0	
	\end{equation}
	or
	\begin{equation}
		H(\bar{\theta}^{(N+1)}_{\eta'_{2}} - \bar{\theta}^{(N+1)}_{\eta_{2}}) - H(\bar{\theta}^{(N)}_{\eta'_{2}} - \bar{\theta}^{(N)}_{\eta_{2}}) \leq 0	
	\end{equation}
	must hold for a nearest-neighbour of either $\eta_1$ or $\eta_2$. Let us denote $\eta_3$ to be this index. Notice that $\eta_3 \neq \eta_1,\eta_2$ since the coupling terms between these neighbouring cells has been eliminated by the odd symmetry of the coupling function. Now for $i = 1$ or $2$ we have that
	\begin{equation}
		H(\bar{\theta}^{(N+1)}_{\eta_{3}} - \bar{\theta}^{(N+1)}_{\eta_{i}}) - H(\bar{\theta}^{(N)}_{\eta_{3}} - \bar{\theta}^{(N)}_{\eta_{i}}) \leq 0,
	\end{equation}
	then from the argument laid out for $\eta_2$ above, we have that
	\begin{equation}
		\bar{\theta}^{(N+1)}_{\eta_{3}} - \bar{\theta}^{(N+1)}_{\eta_{i}} \leq \bar{\theta}^{(N)}_{\eta_{3}} - \bar{\theta}^{(N)}_{\eta_{i}} \implies \bar{\theta}^{(N+1)}_{\eta_{3}} < \bar{\theta}^{(N)}_{\eta_{3}},  	
	\end{equation}
	simply by recalling that $\bar{\theta}^{(N+1)}_{\eta_{i}} < \bar{\theta}^{(N)}_{\eta_{i}}$ by definition of $\eta_1$ and $\eta_2$. Finally, the choice of $\eta_3$ is restricted to those indices $(i,j)$ of the reduced system such that neither $j =0$ nor $i = j$, by the previous discussion for the possibilities for $\eta_2$.

We continue this process inductively by considering the differential equations 
\begin{equation}
	\sum_{k=1}^m \dot{\theta}^{(N+1)}_{\eta_{k}} - \sum_{k=1}^m\dot{\theta}^{(N)}_{\eta_{k}}
\end{equation}
evaluated at the respective equilibrium points, for any $m\geq 1$. Since the differential equation of the element indexed by $\eta_1$ is always considered in this sum we will always have a nonnegative term in $H(\bar{\theta}^{(N+1)}_{N+1,j_0} - \bar{\theta}^{(N+1)}_{N,j_0})$, allowing us to determine that there is a nearest-neighbour, $\eta_{m+1}$, of one of the $\eta_k$ such that $\bar{\theta}^{(N)}_{\eta_{m+1}} > \bar{\theta}^{(N+1)}_{\eta_{m+1}}$, via the same process outlined for $m = 1,2$. As before, the sum will eliminate any coupling present between neighbouring indices, meaning that at each step the cardinality of the set of indices $\{\eta_k\}_{k=1}^m$ increases by one, and it can never be the case that $\eta_{m+1} = (i,j)$ is such that $j = 0$ or $i = j$. In this way we are restricted in our choices to those which are contained within the reduced system. This process is illustrated in Figure $\ref{fig:IncreasingAlgorithm}$ for visual reference. 	

	\begin{figure} 
		\centering
		\includegraphics[width = 7cm]{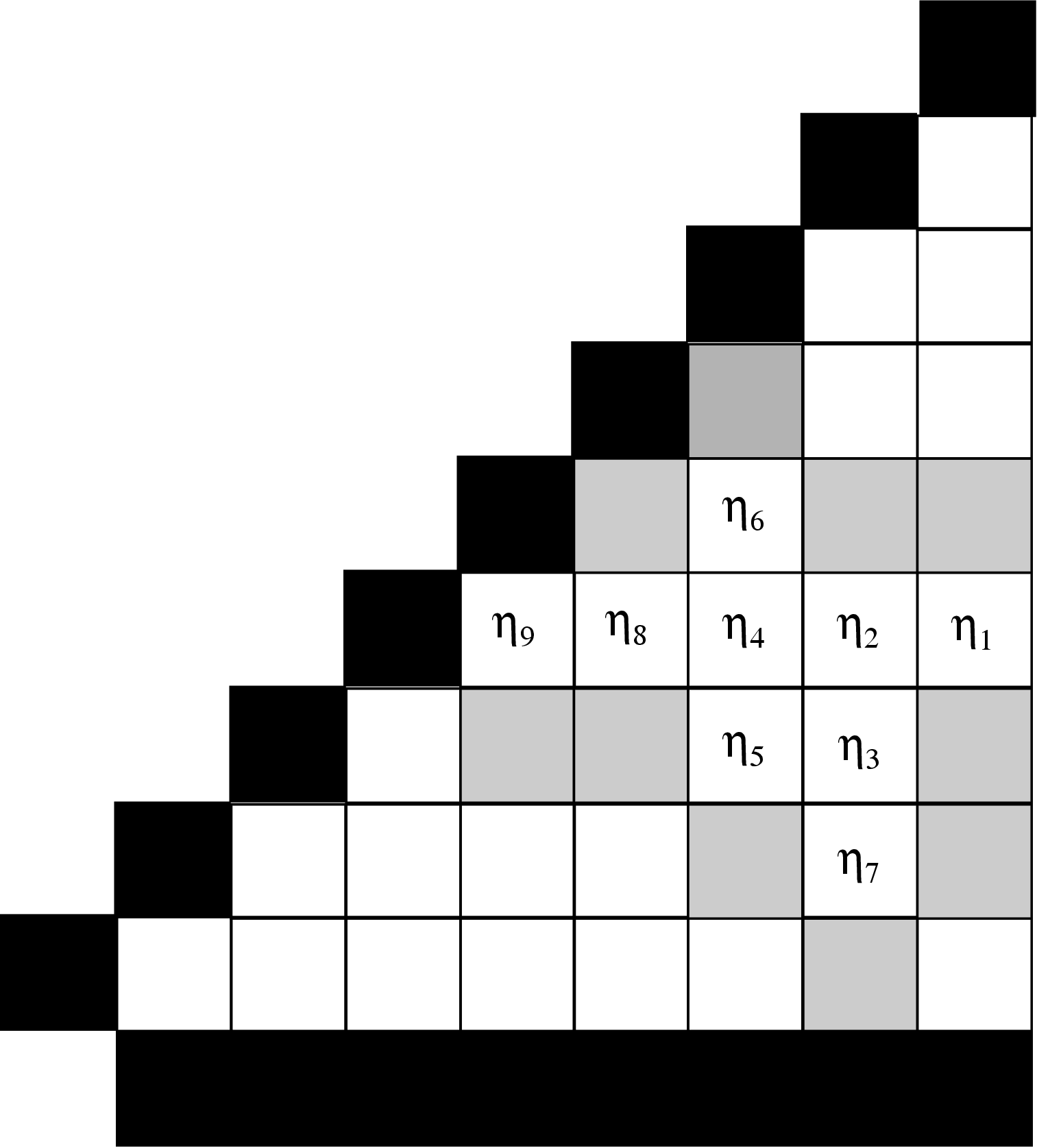}
		\caption{The reduced system with $N = 8$ and a possible collection of the indices $\{\eta_k\}_{k=1}^9$ with the elligible choices for $\eta_{10}$ shaded in. The black cells represent the boundaries of the reduced system which cannot be included in the sequence $\{\eta_k\}_{k=1}^{N(N+2)/2}$.}
		\label{fig:IncreasingAlgorithm}
	\end{figure}  

	We have already noted that the choices of the $\eta_k$ are restricted to those in the reduced system, which is finite. Therefore, this process eventually terminates showing that for every element of the reduced system we have $\bar{\theta}^{(N)}_{i,j} > \bar{\theta}^{(N+1)}_{i,j}$. At this final step we then consider 
	\begin{equation}
		\sum_{k=1}^{N(N-1)/2} \dot{\theta}^{(N+1)}_{\eta_{k}} - \sum_{k=1}^{N(N-1)/2}\dot{\theta}^{(N)}_{\eta_{k}}
	\end{equation} 
	evaluated at the respective equilibria of that lattice size to see that only the interactions with the boundaries of the reduced system remain. That is, we obtain 
	\begin{equation} \label{ExteriorCouplingOnly}
		\begin{split}
		\begin{aligned}
			0 = &(H(\bar{\theta}^{(N)}_{2,1}) - H(\bar{\theta}^{(N+1)}_{2,1})) + 2\sum_{k=2}^{N-1} \bigg(H(\bar{\theta}^{(N)}_{k,k-1}) - H(\bar{\theta}^{(N+1)}_{k,k-1})\bigg)\\ 
			&+ (H(\bar{\theta}^{(N)}_{N,N-1}) - H(\bar{\theta}^{(N+1)}_{N,N-1})) + \sum_{k=1}^{N}\bigg(H(\frac{\pi}{2} - 2\bar{\theta}^{(N+1)}_{k,1}) - H(\frac{\pi}{2} - 2\bar{\theta}^{(N)}_{k,1})\bigg)\\
			&+ \sum_{k=1}^{N}H(\bar{\theta}^{(N+1)}_{N+1,k} - \bar{\theta}^{(N+1)}_{N,k}),
		\end{aligned}
		\end{split}
	\end{equation} 
	where the first three groups of terms come from the coupling with the diagonal $(i,i)$ terms, the fourth grouping coming from the coupling with the $j=0$ row and the final grouping of terms is from the coupling to the right of the $N$th column. We have used the special forms $(\ref{SpecialForm1})$ and $(\ref{SpecialForm2})$ of the boundary interactions to obtain these reductions. Now from our inductive proceedure we have shown that $\bar{\theta}^{(N)}_{i,j} > \bar{\theta}^{(N+1)}_{i,j}$ for every element of the reduced system, and from the fact that $H$ is increasing on $(\frac{-\pi}{2},\frac{\pi}{2})$ we have that 
	\begin{equation}
		\begin{split}
		\begin{aligned}
			&H(\bar{\theta}^{(N)}_{2,1}) - H(\bar{\theta}^{(N+1)}_{2,1}) > 0, \\
			&\sum_{k=2}^{N-1} \bigg(H(\bar{\theta}^{(N)}_{k,k-1}) - H(\bar{\theta}^{(N+1)}_{k,k-1})\bigg) > 0, \\
			&H(\bar{\theta}^{(N)}_{N,N-1}) - H(\bar{\theta}^{(N+1)}_{N,N-1}) > 0, \\
			&\sum_{k=1}^{N}\bigg(H(\frac{\pi}{2} - 2\bar{\theta}^{(N+1)}_{k,1}) - H(\frac{\pi}{2} - 2\bar{\theta}^{(N)}_{k,1})\bigg) > 0.	
		\end{aligned}
		\end{split}
	\end{equation}
	Furthermore, 
	\begin{equation}
		\sum_{k=1}^{N}H(\bar{\theta}^{(N+1)}_{N+1,k} - \bar{\theta}^{(N+1)}_{N,k}) \geq 0	
	\end{equation}
	from our results in Lemma $\ref{lem:Horizontal}$, showing that the right hand side of $(\ref{ExteriorCouplingOnly})$ is strictly positive, which is impossible. Therefore, we have derived a contradiction, thus proving the lemma.    	
	\end{proof} 
	
	Now that we have proven that $i_0 \neq N$, we will move one column to the left and consider the case when $i_0 = N -1$.
	
	\begin{lemma}[Case $i_0 = N-1$] \label{lem:i_0=N-1} 
	Let us denote $\bar{\theta}^{(N)}_{i,j}$ as the solutions of the finite $2N\times 2N$ lattice in the reduced system for any $N \geq 2$. Then $\bar{\theta}^{(N)}_{N-1,j} \leq \bar{\theta}^{(N+1)}_{N-1,j}$ for all $1 \leq j < N-1$.
	\end{lemma}
	
	\begin{proof}
	Let us proceed as in a similar manner to that of the proof of Lemma~\ref{lem:i_0=N}. That is, let us fix $N \geq 2$ and assume that there exists an index $(N-1,j_0)$ of the reduced system such that $\bar{\theta}^{(N)}_{N-1,j_0} > \bar{\theta}^{(N+1)}_{N-1,j_0}$. We will again obtain a contradiction to this assumption, and in turn prove the lemma.
	
	We begin by noting that if $\bar{\theta}^{(N)}_{N-1,j_0} > \bar{\theta}^{(N+1)}_{N-1,j_0}$ then necessarily 
	\begin{equation} \label{N-1Condition}
		H(\bar{\theta}^{(N+1)}_{N,j_0} - \bar{\theta}^{(N+1)}_{N-1,j_0}) - H(\bar{\theta}^{(N)}_{N,j_0} - \bar{\theta}^{(N)}_{N-1,j_0}) \geq 0. 
	\end{equation} 
	Indeed, if we assume that this is not true then it is the case that
	\begin{equation}
		\bar{\theta}^{(N+1)}_{N,j_0} - \bar{\theta}^{(N+1)}_{N-1,j_0} < \bar{\theta}^{(N)}_{N,j_0} - \bar{\theta}^{(N)}_{N-1,j_0} \implies \bar{\theta}^{(N)}_{N,j_0} > \bar{\theta}^{(N+1)}_{N,j_0}, 	
	\end{equation}
which Lemma~\ref{lem:i_0=N} has already shown to be impossible. With this in mind we can proceed as in the proof of Lemma~\ref{lem:i_0=N} by considering the differential equations 
\begin{equation}
	\dot{\theta}_{N-1,j_0}^{(N+1)} - \dot{\theta}_{N-1,j_0}^{(N)}
\end{equation}
evaluated at their respective equilibria. This gives
 \begin{equation} \label{N-1j0Diff}
		\begin{aligned}
		\begin{split}
			0 =\ &H(\bar{\theta}^{(N+1)}_{N-1,j_0+1} - \bar{\theta}^{(N+1)}_{N-1,j_0}) - H(\bar{\theta}^{(N)}_{N-1,j_0+1} - \bar{\theta}^{(N)}_{N-1,j_0}) \\
			& + H(\bar{\theta}^{(N+1)}_{N-1,j_0-1} - \bar{\theta}^{(N+1)}_{N,j_0}) - H(\bar{\theta}^{(N)}_{N-1,j_0-1} - \bar{\theta}^{(N)}_{N-1,j_0}) \\
			& + H(\bar{\theta}^{(N+1)}_{N-2,j_0} - \bar{\theta}^{(N+1)}_{N-1,j_0}) - H(\bar{\theta}^{(N)}_{N-2,j_0} - \bar{\theta}^{(N)}_{N-1,j_0}) \\
			& + H(\bar{\theta}^{(N+1)}_{N,j_0} - \bar{\theta}^{(N+1)}_{N-1,j_0}) - H(\bar{\theta}^{(N)}_{N,j_0} - \bar{\theta}^{(N)}_{N-1,j_0}). 
		\end{split}
		\end{aligned}
	\end{equation}
One will notice a slight difference to the case when $i_0 = N$ in that now we have an even number of terms in the equation, but from $(\ref{N-1Condition})$ we may rearrange $(\ref{N-1j0Diff})$ to be handled in a similar way to the case when $i_0 = N$ by
 \begin{equation}
		\begin{aligned}
		\begin{split}
			0 \geq\ &H(\bar{\theta}^{(N+1)}_{N-1,j_0+1} - \bar{\theta}^{(N+1)}_{N-1,j_0}) - H(\bar{\theta}^{(N)}_{N-1,j_0+1} - \bar{\theta}^{(N)}_{N-1,j_0}) \\
			& + H(\bar{\theta}^{(N+1)}_{N-1,j_0-1} - \bar{\theta}^{(N+1)}_{N,j_0}) - H(\bar{\theta}^{(N)}_{N-1,j_0-1} - \bar{\theta}^{(N)}_{N-1,j_0}) \\
			& + H(\bar{\theta}^{(N+1)}_{N-2,j_0} - \bar{\theta}^{(N+1)}_{N-1,j_0}) - H(\bar{\theta}^{(N)}_{N-2,j_0} - \bar{\theta}^{(N)}_{N-1,j_0}).
		\end{split}
		\end{aligned}
	\end{equation}
Thus the term $(\ref{N-1Condition})$ acts as the single nonnegative term in the difference when $i_0 = N$. Proceeding as above we can find a nearest-neighbour of $(N-1,j_0)$ to which the element at that index for the $2N\times 2N$ lattice is greater than the element at that index for the $2(N+1)\times 2(N+1)$ lattice. This leads to the same chain of steps as in the case when $i_0 = N$, where we now find that the indices which give the desired conditions are limited to not only those in the reduced system, but to those not in the $N$th column. 

Again this procedure terminates in a finite number of steps, showing that $\bar{\theta}_{i,j}^{(N+1)} < \bar{\theta}_{i,j}^{(N)}$ for all $1 \leq j < i \leq N-1$. Then as above, the differential equations
\begin{equation}
	\sum_{1 \leq j < i \leq N-1} \dot{\theta}^{(N+1)}_{i,j} - \sum_{1 \leq j < i \leq N-1}\dot{\theta}^{(N)}_{i,j}
\end{equation} 
evaluated at their respective equilibria leads to a cancelling of all interior coupling terms, thus leaving only connections with the boundary:
	\begin{equation} \label{ExteriorCouplingOnly2}
		\begin{split}
		\begin{aligned}
			0 = &(H(\bar{\theta}^{(N)}_{2,1}) - H(\bar{\theta}^{(N+1)}_{2,1})) + 2\sum_{k=2}^{N-2} \bigg(H(\bar{\theta}^{(N)}_{k,k-1}) - H(\bar{\theta}^{(N+1)}_{k,k-1})\bigg)\\ 
			&+ (H(\bar{\theta}^{(N)}_{N-1,N-2}) - H(\bar{\theta}^{(N+1)}_{N-1,N-2})) + \sum_{k=1}^{N-1}\bigg(H(\frac{\pi}{2} - 2\bar{\theta}^{(N+1)}_{k,1}) - H(\frac{\pi}{2} - 2\bar{\theta}^{(N)}_{k,1})\bigg)\\
			&+ \sum_{k =1}^{N-1} \bigg(H(\bar{\theta}^{(N+1)}_{N,k} - \bar{\theta}^{(N+1)}_{N-1,k}) - H(\bar{\theta}^{(N)}_{N,k} - \bar{\theta}^{(N)}_{N-1,k})\bigg).
		\end{aligned}
		\end{split}
	\end{equation} 
In exactly the same way as the case $i_0 = N$ we have that 	
	\begin{equation}
		\begin{split}
		\begin{aligned}
			&H(\bar{\theta}^{(N)}_{2,1}) - H(\bar{\theta}^{(N+1)}_{2,1}) > 0, \\
			&\sum_{k=2}^{N-2} \bigg(H(\bar{\theta}^{(N)}_{k,k-1}) - H(\bar{\theta}^{(N+1)}_{k,k-1})\bigg) > 0, \\
			&H(\bar{\theta}^{(N)}_{N-1,N-2}) - H(\bar{\theta}^{(N+1)}_{N-1,N-2}) > 0, \\
			&\sum_{k=1}^{N-1}\bigg(H(\frac{\pi}{2} - 2\bar{\theta}^{(N+1)}_{k,1}) - H(\frac{\pi}{2} - 2\bar{\theta}^{(N)}_{k,1})\bigg) > 0.	
		\end{aligned}
		\end{split}
	\end{equation}
Furthermore, following $(\ref{N-1Condition})$ one has that
\begin{equation}
	\sum_{k =1}^{N-1} \bigg(H(\bar{\theta}^{(N+1)}_{N,k} - \bar{\theta}^{(N+1)}_{N-1,k}) - H(\bar{\theta}^{(N)}_{N,k} - \bar{\theta}^{(N)}_{N-1,k})\bigg) \geq 0,	
\end{equation}	
therefore, upon putting this all together, the right hand side of $(\ref{ExteriorCouplingOnly2})$ is strictly positive. This gives a contradiction, thus showing that $i_0 \neq N-1$. 
\end{proof}

Having now proven Lemmas~\ref{lem:i_0=N} and \ref{lem:i_0=N-1}, we have that the remaining cases are quite similar. Therefore, we state these remaining cases as the proof of Lemma~\ref{lem:IncreasingWithLattice}.

\begin{proof}[Proof of Lemma~\ref{lem:IncreasingWithLattice}]

Fix $N \geq 2$ and assume there exists an index $(i_0,j_0)$ of the reduced system $(\ref{ReducedFinitePhase})$ such that $\bar{\theta}^{(N)}_{i_0,j_0} > \bar{\theta}^{(N+1)}_{i_0,j_0}$. From Lemmas~\ref{lem:i_0=N} and \ref{lem:i_0=N-1} we have that $i_0 \neq N,N-1$. We then continue with the method of proof used in Lemma~\ref{lem:i_0=N-1} by showing that if $i_0 \neq N - k$, for some $1 \leq k < N$, then $i_0 \neq N- k -1$ by merely applying the same arguments which were used in proving that $i_0 \neq N -1$ from the result that $i_0 \neq N$. This leads to a process of systematically decreasing $i_0$ by one each step and repeating a similar argument used in the cases $i_0 = N, N-1$ to see that there is no such column in the reduced system which can contain an element that satisfies $\bar{\theta}^{(N)}_{i,j} > \bar{\theta}^{(N+1)}_{i,j}$. This then completes the proof of the lemma since there is only a finite number of columns to check.   
\end{proof} 

This leads to the proof of Theorem $\ref{thm:InfinitePhase}$.

\begin{proof}[Proof of Theorem $\ref{thm:InfinitePhase}$]
	Lemma $\ref{lem:IncreasingWithLattice}$	shows that by observing the value of the equilibrium solution for each lattice size at a single index in the reduced system we form an increasing sequence. By taking any $N\geq 2$ we can identify the equilibrium solution inside the reduced system of the $2N\times 2N$ lattice as an element of the reduced system in the infinite lattice by appending the elements
	\begin{equation}
		\bar{\theta}^{(N)}_{i,j} = 0
	\end{equation}
	for all $i > N$ and $1 \leq j < i$. We have now created a sequence of elements in the reduced system of the infinite lattice which is pointwise increasing and each element is bounded above by $\pi/4$, therefore this sequence converges pointwise as $N \to \infty$. 
	
	Let us write
	\begin{equation}
		\bar{\theta}_{i,j} := \lim_{N \to \infty} \bar{\theta}^{(N)}_{i,j},\ \ \ \ \ 1\leq j < i.
	\end{equation}  
	By the continuity of the differential equations at each index of the reduced system we see that these elements are themselves an equilibrium of the reduced system of the infinite lattice and that $\bar{\theta}_{i,j} \in (0,\frac{\pi}{4}]$ for all $1 \leq j < i$. Moreover, we can apply the symmetries of the finite lattices shown in Figure $\ref{fig:ErmentroutDiagram}$ to extend this equilibrium in the reduced system to an equilibrium of the entire two-dimensional lattice via the same extensions outlined by Ermentrout and Paullet in $(\ref{SymmetryExtensions})$. This therefore proves Theorem $\ref{thm:InfinitePhase}$.
\end{proof}

\begin{rmk}
It was noted in \cite{ErmentroutSpiral} that a rotating wave solution on the lattice is such that the phase-lags over any concentric ring about the centre four cell ring increase from $0$ up to $2\pi$. Although not explicitly stated in our result, this is indeed the case. Furthermore, such a result was implied in Ermentrout and Paullet's work on the finite lattice, although it is notably absent from their work.  We state the following lemma without proof.
\begin{lemma} \label{lem:Vertical} 
	Let $N \geq 2$ and finite. If we denote $\bar{\theta}_{i,j}$ as the solutions on the finite $2N\times 2N$ lattice in the reduced system, then $\bar{\theta}_{i,j} \geq \bar{\theta}_{i,j+1}$ for all $1 \leq j < i \leq N$.   	
\end{lemma}	
\noindent The proof of Lemma $\ref{lem:Vertical}$ is carried out in an almost identical process to that of Lemma $\ref{lem:Horizontal}$ and is therefore omitted from this work. Letting $N \to \infty$ we see that the inequalities along both the rows and columns given by Lemmas $\ref{lem:Horizontal}$ and $\ref{lem:Vertical}$, respectively, remain true and once the extensions from the reduced system to the whole lattice are applied we obtain a true rotating wave solution.    
\end{rmk}

\section{Persistence Into $\alpha > 0$} \label{sec:Persistence}

We now return to the full system $(\ref{ComplexLattice})$. Upon reintroducing the amplitude component to the equations, it was pointed out in Section $\ref{sec:Model}$ that by solving the equations
\begin{equation} \label{0Eqns2}
	\begin{split}
	\begin{aligned}
		&0 = \alpha\sum_{i',j'} (r_{i',j'}\cos(\theta_{i',j'} - \theta_{i,j}) - r_{i,j}) + r_{i,j}(1 - r_{i,j}^2), \\
		&0 =  \sum_{i',j'} \frac{r_{i',j'}}{r_{i,j}}\sin(\theta_{i',j'} - \theta_{i,j}), 
	\end{aligned}
	\end{split}		
\end{equation}   
for the variables $r = \{r_{i,j}\}_{(i,j)\in\mathbb{Z}^2}$ and $\theta = \{\theta_{i,j}\}_{(i,j)\in\mathbb{Z}^2}$ we can obtain solutions to the Ginzburg-Landau system $(\ref{ComplexLattice})$ which are oscillating with a frequency of $2\pi/\omega$. Therefore, the solution to the phase equations $(\ref{PhaseLDS})$ given in Section $\ref{sec:InfiniteLattice}$ provides a rotating wave solution when $\alpha = 0$, and we now wish to extend this solution into $\alpha > 0$. Typically this would be achieved via an Implicit Function Theorem argument, but there are some important technical hurdles which must be addressed before a direct application of the theorem. In this section we will outline some of the largest of these technical hurdles and dedicate the sequel to this work to overcoming these issues. Some details will be left out of this section since we are only attempting to show the reader why the problem of extending our solution into positive $\alpha$ is not a straightforward application of the Implicit Function Theorem.  

We will begin by connecting our work with that undertaken for finite square lattices, as well as formulate a conjecture into the persistence of the solution in the anti-continuum limit of $(\ref{ComplexLattice})$. Following this brief discussion we turn to the major technical issues which must first be addressed if one wishes to obtain a persistence result for rotating waves in the case when $\alpha > 0$ is sufficiently small. In the second part of this work we will address these technical issues and demonstrate exactly how one overcomes them.

\subsection{Connection With the Finite Lattice}

In $\cite{ErmentroutLambdaOmega}$ the existence of rotating wave solutions to systems of type $(\ref{ComplexLattice})$ on finite square lattices are proven. In this case the interactions with boundary elements are absent from the sum notation containing nearest-neighbour interactions, similar to the finite phase system $(\ref{FinitePhase})$. In this case the boundary conditions are said to reflect Neumann boundary conditions coming from the model prior to spatial discretization. The authors show that the solution originating at $\alpha = 0$ cannot persist for all $\alpha > 0$ and must meet the trivial solution at a bifurcation point. Here we will briefly summarize how these results could help to understand the persistence of our solution off of the anti-continuum limit. 

It is shown that upon linearizing a finite square lattice system analogous to $(\ref{ComplexLattice})$ with these Neumann boundary conditions about the trivial equilibrium leads to eigenvalues with real parts given by $1 + \alpha \nu$, where $\nu \leq 0$ is an eigenvalue of the discretized Laplacian operator on the finite lattice. Clearly any bifurcation must take place when $1 + \alpha \nu = 0$, thus we can solve for the values of $\alpha$ which lead to bifurcations from the trivial equilibrium. One always has that $\nu = 0$ is an eigenvalue, but in this case we cannot trigger a bifurcation by varying $\alpha$. The second smallest eigenvalue is given by
\begin{equation}
	\nu = 2\bigg[\cos\bigg(\frac{\pi}{N}\bigg) - 1\bigg], 
\end{equation}  
where we assume the square lattice has $N\times N$ elements. Hence, the minimal value of $\alpha$ which can lead to a bifurcation from the trivial equilibrium is given as a function of the size of the lattice by
\begin{equation}
	\alpha^*(N) = \frac{1}{2\bigg[1 - \cos\bigg(\frac{\pi}{N}\bigg)\bigg]}.
\end{equation}
This value $\alpha^*(N)$ guarantees a minimal range of existence of a rotating wave solution given by $\alpha \in [0,\alpha^*(N)]$. One can further see that $\alpha^*(N) \to \infty$ monotonically as $N\to \infty$, and naturally leads to the following conjecture regarding the persistence of our rotating wave solution found in this work. 

\begin{conj}
	The rotating wave solution to $(\ref{ComplexLattice})$ obtained in this work for $\alpha = 0$ persists for all $\alpha > 0$. 
\end{conj}

Although the work on the finite lattice is quite convincing of the persistence of the solutions into $\alpha > 0$, it by no means should be substituted for a complete analytical investigation. It should be pointed out that the work of $\cite{ErmentroutLambdaOmega}$ benefits greatly from the fact that the system is finite-dimensional and thus is not presented with some technical problems unique to the infinite-dimensional setting studied herein. In the following subsection we detail exactly what some of these technicalities are.

\subsection{Technical Hurdles of the Infinite-Dimensional Lattice}  

To begin, solving $(\ref{0Eqns2})$ can be interpreted abstractly as obtaining roots to the mapping $F = (F^1,F^2)^T$ of the form
\begin{equation} \label{FMapping}
	\begin{split}
		F_{i,j}^1(\alpha, r,\theta) = \alpha\sum_{i',j'}& [r_{i',j'}\cos(\theta_{i',j'} - \theta_{i,j}) - r_{i,j}] + r_{i,j}(1 - r^2_{i,j}), \\
		F_{i,j}^2(\alpha, r,\theta) = \sum_{i',j'}& \frac{r_{i',j'}}{r_{i,j}}\sin(\theta_{i',j'} - \theta_{i,j}),
	\end{split}
\end{equation} 
where $(i,j) \in \mathbb{Z}^2$. Moreover, denoting ${\bf 1} = \{1\}_{(i,j)\in\mathbb{Z}^2}$ and $\bar{\theta} = \{\theta_{i,j}\}_{(i,j)\in\mathbb{Z}^2}$ to be the rotating wave solution guaranteed by Theorem $\ref{thm:InfinitePhase}$ we have $F(0, {\bf 1},\bar{\theta}) = 0$. Therefore, using $\alpha$ as a system parameter to apply the Implicit Function Theorem we require the following two ingredients:
\begin{itemize}
	\item $F$ is a well-defined, smooth mapping between Banach spaces,
	\item The Fr\'echet derivative of $F$ with respect to $(r,\theta)$ at the point $(\alpha, r, \theta) = (0, {\bf 1},\bar{\theta})$ is an isomorphism of Banach spaces.
\end{itemize} 
Let us now discuss the various ways in which these conditions can fail to be met with the mapping $(\ref{FMapping})$.  

The natural Banach space setting which one would employ to examine $(\ref{FMapping})$ would be the spaces of sequences indexed by $\mathbb{Z}^2$. The typically studied examples of such spaces are denoted $\ell^p(\mathbb{Z}^2)$ and are given by    
\begin{equation} \label{ellpSpace}
	\ell^p(\mathbb{Z}^2) = \{ \{x_{i,j}\}_{(i,j)\in\mathbb{Z}^2}\ |\ \sum_{(i,j)\in \mathbb{Z}^2} |x_{i,j}|^p < \infty\},	
\end{equation}  
for any $p\in [1,\infty)$ and $\ell^\infty(\mathbb{Z}^2)$ the space of all bounded sequences indexed by $\mathbb{Z}^2$. Something that is key to our work here is that no two $\ell^p(\mathbb{Z}^2)$ spaces are isomorphic unless they are the same space. One should also note that any sequence belonging to $\ell^p(\mathbb{Z}^2)$ for $1 \leq p < \infty$ must exhibit some algebraic decay as $|i| + |j| \to \infty$, as this will be crucial to our understanding of the mapping $F$.   

Now, from the form of the mapping $(\ref{FMapping})$ for any $\alpha \in \mathbb{R}$ we have 
\begin{equation}
	\begin{split}
		F_{i,j}^1(\alpha, {\bf 1},\bar{\theta}) = \alpha\sum_{i',j'}& [\cos(\bar{\theta}_{i',j'} - \bar{\theta}_{i,j}) - 1], \\
		F_{i,j}^2(\alpha, {\bf 1},\bar{\theta}) = 0.
	\end{split}
\end{equation} 
Notice that for $\alpha \neq 0$, we have that $F^1(\alpha, {\bf 1},\bar{\theta})$ does not necessarily vanish. Hence, without having an efficient estimate on the decay of $|\bar{\theta}_{i',j'} - \bar{\theta}_{i,j}|$ as $|i| + |j| \to \infty$, one cannot obtain any neighbourhood of $\alpha = 0$ in which $F^1(\alpha, {\bf 1},\bar{\theta})$ maps into $\ell^p(\mathbb{Z}^2)$ for $1 \leq p < \infty$. This only leaves the possibility of working with $\ell^\infty(\mathbb{Z}^2)$.

One can see that the partial Fr\'echet derivative of $F^1$ with respect to $\theta$ evaluated at the point $(\alpha, r, \theta) = (0, {\bf 1},\bar{\theta})$ will vanish due to the linear dependence of the coupling terms on $\alpha$, and therefore the full Fr\'echet derivative of $F$ with respect to $(r,\theta)$, denoted $DF$, can be interpreted as a lower-triangular matrix of linear operators. Hence, a necessary condition for $DF$ to be a Banach space isomorphism is that the partial Fr\'echet derivative of $F^1$ with respect to $r$ at the point $(\alpha, r, \theta) = (0, {\bf 1},\bar{\theta})$ be a Banach space isomorphism. Moreover, coupling the previous arguments to make $F^1$ a well-defined operator between sequence spaces with the fact that no two $\ell^p(\mathbb{Z}^2)$ spaces are isomorphic unless they are the same space necessitates that we have the image of $F^1$ and the domain of $r$ be $\ell^\infty(\mathbb{Z}^2)$. We now discuss the limitations of this requirement. 

With $r$ belonging to the space $\ell^\infty(\mathbb{Z}^2)$ we arrive at a similar condition on $F^2$ to be well-defined as we had encountered for $F^1$. That is, for any $r \in \ell^\infty(\mathbb{Z}^2)$ with nonzero elements we have
\begin{equation}
	F_{i,j}^2(\alpha, r,\bar{\theta}) = \sum_{i',j'} \frac{r_{i',j'}}{r_{i,j}}\sin(\bar{\theta}_{i',j'} - \bar{\theta}_{i,j}),	
\end{equation}
where we note the independence of $F^2$ with respect to $\alpha$. When $r$ is not taken to be a constant sequence we find ourselves in a similar position to the case of $F^1$ since we require an understanding of the decay of $|\bar{\theta}_{i',j'} - \bar{\theta}_{i,j}|$ as $|i| + |j| \to \infty$ in order to have $F^2$ map into $\ell^p(\mathbb{Z}^2)$ for $1 \leq p < \infty$. Hence the image of $F^2$ must also be taken to be $\ell^\infty(\mathbb{Z}^2)$, and the diagonal nature of $DF$ forces the domain of $\theta$ to also be $\ell^\infty(\mathbb{Z}^2)$ for the same reasons as the case of the mapping component $F^1$.   

Now that we have narrowed down the possibilities for Banach spaces in which the mapping $F$ is well-defined, let us now turn to the problem of invertibility of the Fr\'echet derivative. As previously noted, $DF$ can be interpreted as a lower-triangular matrix whose entries are Banach space operators. In this way invertibility of $DF$ entirely falls upon having the diagonal entries of this matrix being continuously invertible Banach space operators (Banach space isomorphisms). The specific problem which will arise in our situation is best understood once we introduce some nomenclature with regards to the spectrum of an operator. A linear operator is said to be a {\em Fredholm operator} if its range is closed and its kernel and cokernel are finite dimensional. Then a complex number $\lambda$ is said to belong to the {\em essential spectrum} of a linear operator $T$ on a Banach space $X$ if $T - \lambda I$ is not a Fredholm operator. Here we use $I$ to denote the identity function on the Banach space $X$. That is, an element of the essential spectrum is an element which is not an `isolated eigenvalue' and hence the non-invertibility cannot be overcome by simply moving to a quotient space.     

Returning to our situation, let us denote $L:\ell^\infty(\mathbb{Z}^2) \to \ell^\infty(\mathbb{Z}^2)$ to be the partial Fr\'echet derivative of $F^2$ with respect to $\theta$ evaluated at the point $(\alpha, r, \theta) = (0, {\bf 1},\bar{\theta})$. Then we have that $L$ acts on the sequences $x = \{x_{i,j}\}_{(i,j)\in\mathbb{Z}^2}$ as
\begin{equation}
	[Lx]_{i,j} = \sum_{i',j'} \cos(\bar{\theta}_{i',j'} - \bar{\theta}_{i,j})(x_{i',j'} - x_{i,j}),
\end{equation} 
for all $(i,j)\in\mathbb{Z}^2$. One should note that all nearest-neighbour interactions are such that $|\bar{\theta}_{i',j'} - \bar{\theta}_{i,j}| \leq \frac{\pi}{2}$, thus making $\cos(\bar{\theta}_{i',j'} - \bar{\theta}_{i,j}) \geq 0$ for all $(i,j)\in\mathbb{Z}^2$. More precisely, with the exception of the 'centre' four cells at $(i,j) = (0,0), (0,1), (1,0), (1,1)$ all nearest-neighbour interactions are such that $|\bar{\theta}_{i',j'} - \bar{\theta}_{i,j}| < \frac{\pi}{2}$, whereas the coupling between any two of the four centre cells is exactly $\pi/2$. This therefore makes $\cos(\bar{\theta}_{i',j'} - \bar{\theta}_{i,j}) > 0$ for all $(i,j), (i',j') \neq (0,0), (0,1), (1,0), (1,1)$. Furthermore, one can see that the translational symmetry of $F$ with respect to $\theta$ endows the operator $L$ with a nontrivial kernel spanned by the constant sequences. Although this makes our operator $L$ not invertible, it has been successfully overcome in the finite dimensional setting $\cite{ErmentroutLambdaOmega}$. Having moved to infinite dimensions the analysis becomes significantly more complicated since one is not necessarily able to quotient out this translational symmetry. The following proposition states that $L$ is not a Fredholm operator, and hence $\lambda = 0$ belongs to the essential spectrum of $L$.

\begin{proposition} \label{prop:ellinfty}
	$L:\ell^\infty(\mathbb{Z}^2) \to \ell^\infty(\mathbb{Z}^2)$ is not a Fredholm operator.
\end{proposition} 

\begin{proof}
	We recall that the norm on $\ell^\infty(\mathbb{Z}^2)$ is given by
	\begin{equation}
		||x||_\infty = \sup_{(i,j)\in\mathbb{Z}^2} |x_{i,j}| < \infty. 
	\end{equation}
	Then to show that $L$ is not a Fredholm operator, we show that it does not have closed range. To do so we show that it is not bounded below by demonstrating that there does not exist a $\delta > 0$ such that $||Lx||_\infty \geq \delta||x||_\infty$ for all $x \in \ell^\infty(\mathbb{Z}^2)$. This is carried out by constructing a sequence $\{x^{(n)}\}_{n=1}^\infty \subset \ell^\infty(\mathbb{Z}^2)$ such that $||x^{(n)}||_\infty = 1$ for all $n\geq 1$ but $||Lx^{(n)}||_\infty \to 0$ as $n \to \infty$. 
	
	The sequence of vectors is constructed in the following way: Let $(i_0,j_0) \in \mathbb{Z}^2$ be an arbitrary index and fix $n\geq 1$. Begin by setting $x^{(n)}_{i_0,j_0} = 0$. For those indices which are one step along the integer lattice $\mathbb{Z}^2$ (nearest-neighbours) to $(i_0,j_0)$ we set the elements of the vector with these indices to $1/n$. Then we set the eight elements which are two steps from the index $(i_0,j_0)$ (nearest-neighbours of the nearest-neighbours) to $2/n$. Next we set the twelve elements which are three steps from the index $(i_0,j_0)$ to $3/n$. We continue this pattern so that for any $k \leq n$ we set those elements which are exactly $k$ steps from the index $(i_0,j_0)$ to $k/n$. For the remaining elements of whose indices lie at more than $n$ steps from the index $(i_0,j_0)$ we set to $1$. Two vectors of this form are shown in Figure $\ref{fig:ellinfty}$ for $n=4,5$ to visualize the form and demonstrate how the vectors change as $n$ increases. 
	
\begin{figure} 
	\centering
		\includegraphics[width = 6cm]{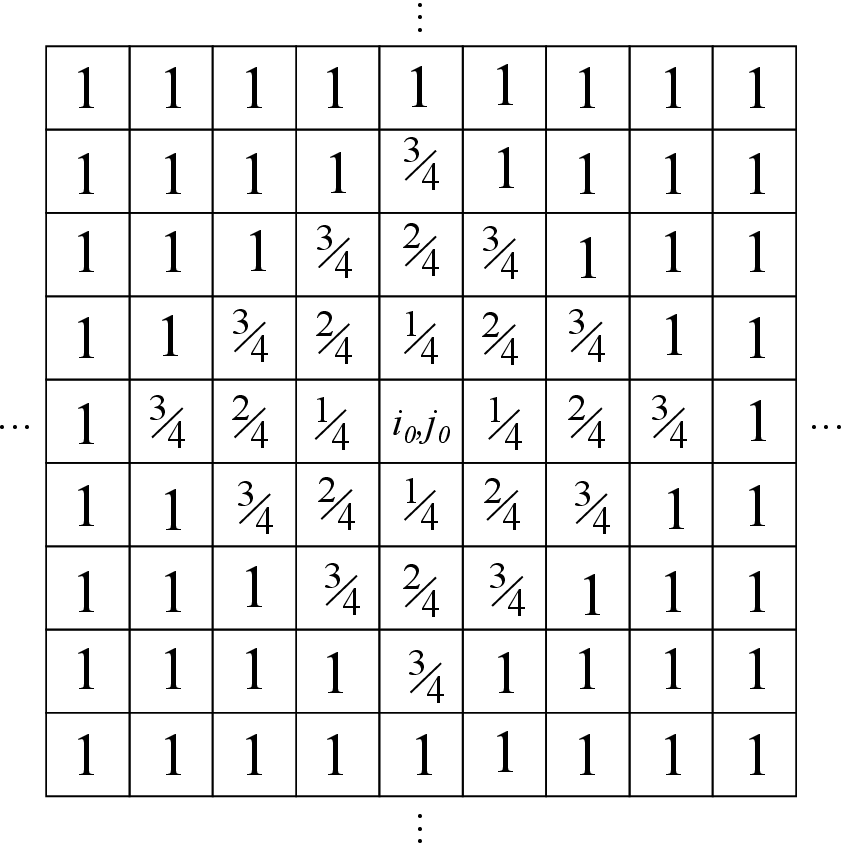}\ \ \ \ \ \
		\includegraphics[width = 6cm]{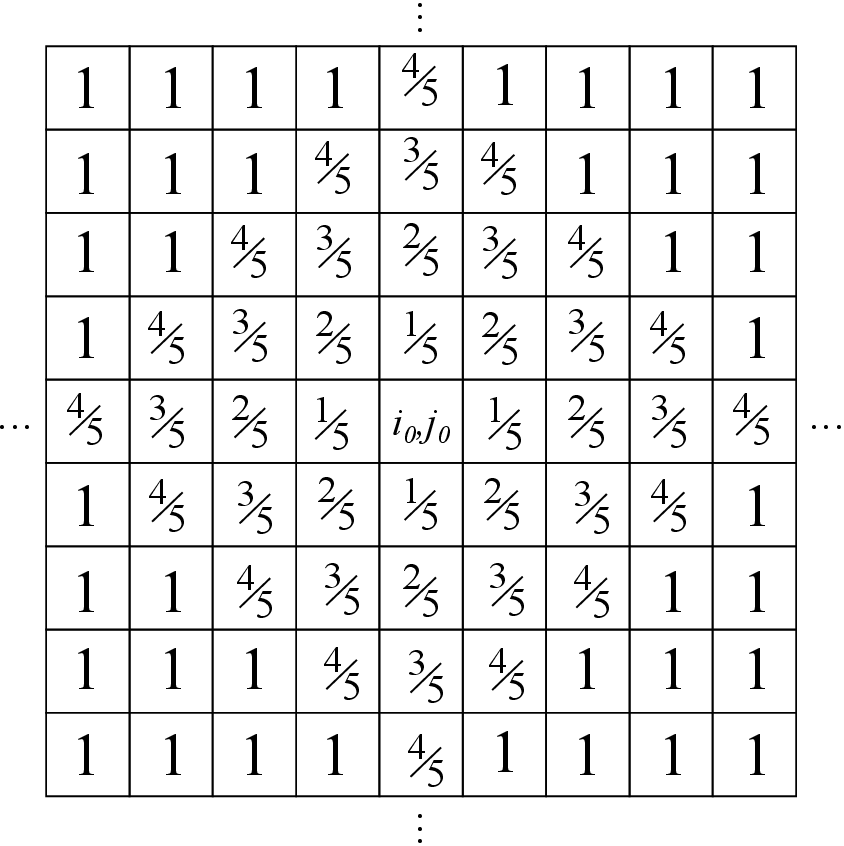}
		\caption{A visualization of two vectors from the sequence $\{x^{(n)}\}_{n=1}^\infty$ from the proof of Proposition $\ref{prop:ellinfty}$. Here we have $n=4,5$ centred at the index $(i_0,j_0)$.}
		\label{fig:ellinfty}	
\end{figure} 
	
	Then clearly for each $n$ this vector has norm $1$ in $\ell^\infty(\mathbb{Z}^2)$, but one should notice that by construction we have
	\begin{equation}
		|x_{i',j'} - x_{i,j}| \leq \frac{1}{n}
	\end{equation}
	for any $(i,j)$ and a nearest-neighbour $(i',j')$ in $\mathbb{Z}^2$. Furthermore,
	\begin{equation} \label{BddBelow}
		|[Lx^{(n)}]_{i,j}| \leq \sum_{(i',j')} |\cos(\bar{\theta}_{i',j'} - \bar{\theta}_{i,j})||x^{(n)}_{i',j'} - x^{(n)}_{i,j}| \leq \frac{4}{n}
	\end{equation}
	for all $(i,j)\in\mathbb{Z}^2$. Therefore taking the supremum of $(\ref{BddBelow})$ over all $(i,j)\in\mathbb{Z}^2$ gives that $||Lx^{(n)}||_\infty \leq 4/n$ and hence $||Lx^{(n)}||_\infty \to 0$ as $n \to \infty$. This shows that $L$ is not bounded below and completes the proof.      
\end{proof} 

With the proof of Proposition $\ref{prop:ellinfty}$ one can now see the various technical hurdles which must be addressed if one wishes to extend the rotating wave solution of this work at $\alpha = 0$ into positive values of $\alpha$. In the sequel to this work we will demonstrate exactly how to overcome these issues by applying a non-standard Implicit Function Theorem to a mapping whose roots lie in one-to-one correspondence with those of $F$ in $(\ref{FMapping})$. This application is highly nontrivial and requires the establishment of several minor results along the way and is therefore left to a subsequent study.

\section*{Acknowledgements} 

This work was supported by an Ontario Graduate Scholarship while at the University of Ottawa. The author is very thankful to Benoit Dionne and Victor LeBlanc for their careful reading of the work, correcting errors and making improvements to properly convey the results.

\end{document}